\documentclass[a4paper, reqno]{amsart} 

\usepackage{enumerate}
\usepackage{color}
\usepackage{indentfirst} 
\usepackage{newlfont}   
\usepackage{verbatim}
\usepackage{bm}
\usepackage{amsthm, amssymb, amsmath, latexsym}
\usepackage{graphicx}
\usepackage{cases}
\usepackage{amsfonts, amscd, amsbsy, mathrsfs}
\usepackage{esint}
\theoremstyle{plain}                     
\begingroup
\newtheorem{teo}{Theorem}[section]       
\newtheorem{prop}[teo]{Proposition}     
        
\newtheorem{lem}[teo]{Lemma}      
\endgroup
      
\theoremstyle{definition}                
\begingroup

\newtheorem{oss}[teo]{Remark}
\endgroup

\numberwithin{equation}{section}
\newcommand{\scal}[2]{\langle #1,\,#2 \rangle}
\newcommand{\myintol}[1]{\int_0^L{#1} \,dx_1}
\newcommand{\myfintol}[1]{\fint_0^L{#1} \,dx_1}
\newcommand{\myintom}[1]{\int_{\Omega}{#1 \,dx_1dsdt}}

\newcommand{\st}[0]{\omega}

\newcommand{\stee}[0]{\omega_{\delta}}

\newcommand{\eh}[0]{\lin{\nabla}_{\frac{\delta_h}{h}}}
\newcommand{\ph}[0]{\Pi_{\frac{\delta_h}{h}}}
\newcommand{\nh}[0]{\nabla_{h,\delta_h}}
\newcommand{\tvh}[0]{\til{v}^h}
\newcommand{\myintss}[1]{\int_{0}^{L}{\int_{0}^{1}{#1\, dsdx_1}}}
\newcommand{\myintr}[1]{\int_{-\frac{1}{2}}^{\frac{1}{2}}{#1\,dt}}
\newcommand{\myintt}[1]{\int_{0}^{1}{#1\,ds}}

\newcommand{\deb}[0]{\rightharpoonup}
\newcommand{\ten}[0]{\rightarrow}

\newcommand{\til}[1]{\widetilde{#1}}

\newcommand{\lin}[1]{\overline{#1}}
\newcommand{\capp}[1]{\widehat{#1}}

\newcommand{\sym}{\mathrm{sym}}

\newcommand{\dist}{\mathrm{dist}}
\newcommand{\comp}{{\,\circ\,}}

\newcommand{\mthree}[0]{\mathbb{M}^{3\times 3}}
\newcommand{\n}[0]{\lin{n}}
\newcommand{\tg}[0]{\lin{\tau}}
\newcommand{\nn}[0]{\lin{\nabla}_{\epsilon}v\lin{R}_0^T}
\newcommand{\nnn}[0]{\lin{\nabla}_{{\epsilon}_j}v^j\lin{R}_0^T}
\newcommand{\nnp}[0]{\lin{\nabla}_{{\epsilon}_j}\phi^j\lin{R}_0^T}
\title[Thin-walled beams with an arbitrary cross-section]{Thin-walled beams with a cross-section of arbitrary geometry: derivation of linear theories starting from 3D nonlinear elasticity} 
 
\author[E. Davoli]{Elisa Davoli} 

\address[E. Davoli]{Scuola Internazionale Superiore di Studi Avanzati, Via Bonomea 265, 34136 Trieste (Italy)}
\email{davoli@sissa.it}

 
\begin{document}

\begin{abstract}
The subject of this paper is the rigorous derivation of lower dimensional models for a nonlinearly elastic thin-walled beam whose cross-section is given by a thin tubular neighbourhood of a smooth curve. Denoting by $h$ and $\delta_h$, respectively, the diameter and the thickness of the cross-section, we analy\-se the case where the scaling factor of the elastic energy is of order $\epsilon_h^2$, with $\epsilon_h/\delta_h^2\ten \ell\in [0,+\infty)$. Different linearized models are deduced according to the relative order of magnitude of $\delta_h$ with respect to $h$. 
\end{abstract} 
\maketitle
\section{Introduction}
\label{introduction}
 \noindent A thin-walled beam is a three-dimensional body, whose length is much larger than the diameter of the cross-section, which, in turn, is much larger than the thickness of the cross-section. This kind of beams are commonly used in mechanical engineering, since they combine good resistance properties with a reasonably low weight. 
 
 In this paper we consider a nonlinearly elastic thin-walled beam with a cross-section of arbitrary geometry and we rigorously deduce, by $\Gamma$-convergence techniques, different lower dimensional linearized models, according to the relative order of magnitude between the cross-section diameter and the cross-section thickness.\\
 
The derivation of lower dimensional models for thin domains is a classical problem in elasticity theory. Since the early 90's a mathematically rigorous approach has emerged \cite{A-B-P, L-R, L-R2}, based on the notion of $\Gamma$- convergence. This variational approach guarantees convergence of minimizers (and of minima) of the three-dimensional energy to minimizers (and minima) of the limit models. Recently, owing to the seminal paper \cite{F-J-M}, hierarchies of limit models have been identified by $\Gamma$-convergence methods for plates \cite{F-J-M, F-J-M2}, shells \cite{F-J-M-M, L-M-P, L-M-P2}, and beams \cite{M-M3, M-M, S0, S}. The different limit models correspond to different scalings of the elastic energy, which, in turn, are determined by the scaling of the applied loads in terms of the thickness parameter.

The subject of this paper is the study of the lower dimensional theories for thin-walled beams that can be deduced by $\Gamma$-convergence from three-dimensional nonlinear elasticity. A similar analysis has been performed in the recent papers \cite{F-M-P, F-M-P2}, in the case of a rectangular cross-section. Here, instead, we assume that the cross-section of the beam is given by a thin tubular neighbourhood of a smooth curve. More precisely, let $\gamma:[0,1]\longrightarrow \mathbb{R}^3$, $\gamma(s)=\gamma_2(s)e_2+\gamma_3(s)e_3$, be a smooth and simple planar curve, whose curvature is not identically equal to zero, and let $n(s)$ denote the normal vector to the curve at the point $\gamma(s)$. We consider an elastic beam of reference configuration $$\Omega_h:=\Big\{x_1e_1+h\gamma(s)+\delta_h t n(s): x_1\in (0,L), s\in (0,1), t\in \Big(-\frac{1}{2},\frac{1}{2}\Big)\Big\},$$
where $L$ is the length of the beam and $h,\delta_h$ are positive parameters. To model a thin-walled beam, we assume 
$$h\ten 0 \quad \text{ and }\quad \frac{\delta_h}{h}\ten 0\quad (\text{as }h\ten 0).$$ In other words, the diameter of the cross-section is of order $h$ and is assumed to be much larger than the cross-section thickness $\delta_h$. 

To any deformation $u\in W^{1,2}(\Omega_h;\mathbb{R}^3)$, we associate the
elastic energy (per unit cross-section) defined by
\begin{equation}
\nonumber
 \cal E^h(u):=\frac{1}{h\delta_h} \int_{\Omega_h}{ W(\nabla u(x))dx},
\end{equation}
where the energy density $W$ satisfies the usual assumptions of nonlinear elasticity (see Section \ref{setting}). We are interested in understanding the behaviour, as $h\ten 0$, of sequences of deformations $(u^h)$ satisfying 
\begin{equation}
\label{ene1}
\cal{E}^h(u^h)\leq C \epsilon_h^2,
\end{equation}
where $(\epsilon_h)$ is a given sequence of positive numbers. Estimate \eqref{ene1} is satisfied, for instance, by global minimizers of the total energy 
$$\cal{E}^h(u)-\frac{1}{h\delta_h}\int_{\Omega_h}{u\cdot f^h dx}$$ 
when the applied body force $f^h:\Omega_h\longrightarrow \mathbb{R}^3$ is of a suitable order of magnitude with respect to $\epsilon_h$ (see \cite{F-M-P, F-M-P2}). The asymptotic behaviour of $(u^h)$, as $h\ten 0$, can be characterized by identifying the $\Gamma$-limit of the sequence of functionals $(\epsilon_h^{-2} \cal{E}^h)$. For the definition and properties of $\Gamma$-convergence we refer to the monograph \cite{DM}.\\ 

In this paper we mainly focus on the case where the sequence $(\epsilon_h)$ is infinitesimal and satisfies 
\begin{equation}
\label{lep}
\lim_{h\ten 0}\frac{\epsilon_h}{\delta_h^2}=:\ell\in [0,+\infty).
\end{equation}
In analogy with the results of \cite{F-M-P2}, this scaling is expected to correspond at the limit to partially or fully linearized models. Other scalings, different than \eqref{lep}, will be studied in a forthcoming paper.
 
 Assuming $\epsilon_h=o(\delta_h)$, as $h\ten 0$,  we first show (Theorem \ref{disp}) that any sequence $(u^h)$ satisfying \eqref{ene1} converges, up to a rigid motion, to the identity deformation on the mid-fiber of the rod; more precisely, defining $\Omega:=(0,L)\times (0,1)\times (-\frac{1}{2}, \frac{1}{2})$ and $\psi^h:\Omega\longrightarrow
\Omega_h$ as
$$\psi^h(x_1,s,t):=x_1e_1+h\gamma(s)+\delta_h t n(s)$$
for every $(x_1,s,t)\in\Omega$, we have that, up to rigid motions,
 $$y^h:=u^h\circ \psi^h\ten x_1e_1$$
strongly in $W^{1,2}(\Omega;\mathbb{R}^3)$. 

 To express the limiting functional, we introduce and study the compactness properties of some linearized quantities associated with the scaled deformations $y^h$. We consider the tangential derivative of the tangential displacement 
$$g^h(x_1,s,t):=\frac{1}{\epsilon_h}\partial_1(y^h_1-x_1)$$
for a.e. $(x_1,s,t)\in \Omega$, and the twist function
$$w^h(x_1,s):=\frac{\delta_h}{h\epsilon_h}\myintr{\partial_s (y^h-\psi^h)\cdot n}$$
for a.e. $(x_1,s)\in (0,L)\times(0,1)$. In Theorem \ref{disp}, under assumption \eqref{lep}, 
we prove that \begin{eqnarray*}
&g^h\deb g &\text{ weakly in }L^2(\Omega),\\
&w^h\ten w &\text{ strongly in }L^2((0,L)\times(0,1)),\\
\end{eqnarray*}
for some $g\in L^2((0,L)\times(0,1))$ and $w\in W^{1,2}(0,L)$. Moreover, the sequence of bending moments $\big(\frac{1}{h}\partial_s w^h\big)$ converges in the following sense:
$$\frac{1}{h}\partial_s w^h\deb b \text{ weakly in }W^{-1,2}((0,L)\times(0,1))$$
for some $b\in L^2((0,L)\times(0,1))$ (see Proposition \ref{fpwb}). In Theorem \ref{deep} we show that the limit quantities $w,g,b$ must satisfy some compatibility conditions that depend on the relative order of magnitude between $\delta_h$ and $h$. More precisely, assuming the existence of the limit
$$\mu :=\lim_{h\ten 0}\frac{\delta_h}{h^3},$$
three main regimes can be identified:
\begin{itemize}
\item $\mu=+\infty,$
\item $\mu\in (0, +\infty)$,
\item $\mu=0.$
\end{itemize}

\

In the first regime $\mu=+\infty$, one has that $g$ is the tangential derivative of the first component of a Bernoulli-Navier displacement in curvilinear coordinates, that is, there exists $v\in W^{1,2}((0,L)\times(0,1);\mathbb{R}^3)$ such that 
$$\partial_1 v \cdot e_1=g, \quad \partial_s v\cdot \tau=0, \quad \partial_s v\cdot e_1+\partial_1 v\cdot \tau=0\quad\text{ on }(0,L)\times(0,1),$$
where $\tau(s)$ denotes the tangent vector to the curve $\gamma$ at the point $\gamma(s)$. The structure of the cross-sectional components of $v$ depends on the existence and the value of the limit 
 $$\lambda := \lim_{h\ten 0}\frac{\delta_h}{h^2}.$$ 
 Indeed, if $\lambda=+\infty$, there exist $\alpha,\beta\in W^{1,2}(0,L)$ such that $$v(x_1,s)\cdot e_2=\alpha(x_1) \quad \text{ and }\quad v(x_1,s)\cdot e_3=\beta(x_1)$$
for every $(x_1,s) \in (0,L)\times(0,1)$. If $\lambda\in (0,+\infty)$, then one can show that the twist function $w$ belongs to $W^{2,2}(0,L)$ and the cross-sectional components of $v$ depend on $w$ in the following way: $$v(x_1,s)\cdot e_2=\alpha(x_1)-\tfrac{1}{\lambda}\,w(x_1)\gamma_3(s) \quad \text{ and }\quad v(x_1,s)\cdot e_3=\beta(x_1)+\tfrac {1}{\lambda}\,w(x_1)\gamma_2(s)$$
for every $(x_1,s) \in (0,L)\times(0,1)$ and for some $\alpha,\beta\in W^{1,2}(0,L)$. Finally, if $\lambda=0$, the twist function $w$ is affine, while the cross-sectional components of $v$ satisfy $$v(x_1,s)\cdot e_2=\alpha(x_1)-\delta(x_1)\gamma_3(s) \quad \text{ and }\quad v(x_1,s)\cdot e_3=\beta(x_1)+\delta(x_1)\gamma_2(s)$$ for every $(x_1,s) \in (0,L)\times(0,1)$ and for some $\alpha, \beta,\delta \in W^{1,2}(0,L)$. In other words, in the regime $\mu=+\infty$, the structure of $g$ is essentially one-dimensional. As for the bending moment $b$, we prove that it simply belongs to $L^2((0,L)\times(0,1))$. 

In the regime $\mu=0$, we still have that $g$ is the tangential derivative of the first component of a Bernoulli-Navier displacement in curvilinear coordinates, but only in an approximate sense (see the definition of the class $\cal{G}$ in Section \ref{limit}). Moreover, the bending moment $b$ is associated with an infinitesimal isometry of the cylindrical surface $$\{x_1e_1+\gamma(s):x_1\in (0,L), s\in (0,1)\},$$
in the sense that there exists $\phi\in L^2((0,L)\times(0,1);\mathbb{R}^3)$, with $\partial_s \phi\in L^2((0,L)\times(0,1);\mathbb{R}^3)$, such that $$\partial_1 \phi\cdot e_1=0,\quad \partial_s \phi\cdot \tau=0, \quad \partial_s \phi\cdot e_1+\partial_1 \phi\cdot \tau=0\quad \text{ on }(0,L)\times(0,1)$$
and
$$\partial_s(\partial_s \phi \cdot n)=b \quad \text{ on }(0,L)\times(0,1).$$
The equalities are intended in the sense of distributions; some higher regularity for $\phi$ can be proved (see Remark \ref{rksyme}). In other words, in this regime the limit kinematic  description of the thin-walled beam is intrinsically two-dimensional.
 
In the intermediate regime $\mu\in (0,+\infty)$, the limit quantities $g$ and $b$ are no more mutually independent but they must satisfy the following constraint: there exists $\phi\in L^2((0,L)\times(0,1);\mathbb{R}^3)$, with $\partial_s \phi\in L^2((0,L)\times(0,1);\mathbb{R}^3)$, such that $$\partial_1 \phi\cdot e_1=\mu g,\quad \partial_s \phi\cdot \tau=0, \quad \partial_s \phi\cdot e_1+\partial_1 \phi\cdot \tau=0\quad \text{ on }(0,L)\times(0,1)$$
and
$$\partial_s(\partial_s \phi \cdot n)=b \quad \text{ on }(0,L)\times(0,1).$$

Finally, for the twist function $w$, we show that it is affine for $\mu\in [0,+\infty)$. 

The $\Gamma$-limit functional is expressed in terms of the limit quantities $w,g,b$ and, according to the values of $\lambda$ and $\mu$, is finite only on the class $\cal{A}_{\lambda,\mu}$ of triples $(w,g,b)$ with the structure described above. In Theorems \ref{liminft} and \ref{limsupt} we prove that for $(w,g,b)\in \cal{A}_{\lambda,\mu}$ the $\Gamma$-limit is given by the functional 
 $$\cal{J}_{\lambda,\mu}(g,w,b)=\frac{1}{24}\myintss{Q_2(s,w',b)}+\frac{1}{2}\myintss{\mathbb{E} g^2},$$ where $Q_2$ is a positive definite quadratic form and $\mathbb{E}$ is a positive constant, for which an explicit formula is provided (see \eqref{defE} and \eqref{conrm}).\\

The proofs of compactness and of the liminf inequality rely on the rigidity estimate due to Friesecke, James and M\"uller (Theorem \ref{fjm}) and on a rescaled two-dimensional Korn's inequality in curvilinear coordinates for cross-sectional displacements (Theorem \ref{Korn}).
The key ingredients in the construction of the recovery sequences are some approximation results for triples in the classes $\cal{A}_{\lambda,\mu}$ in terms of smooth functions (see Section \ref{limit}).
In the regime $\mu=0$ the approximation result is proved under the additional assumption that the set where the curvature of $\gamma$ vanishes is the union of a finite number of intervals and isolated points. Therefore, for $\mu=0$ the $\Gamma$-convergence result is valid only under this additional restriction.\\

The dependence of the $\Gamma$-limits on the rate of convergence of the thickness parameter $\delta_h$ with respect to the cross-section diameter $h$ is an effect of the nontrivial geometry of the cross-section. Indeed, in the case of a rectangular cross-section this phenomenon is not observed for the scalings \eqref{lep} and is conjectured to arise only for scalings $\epsilon_h$ such that $\delta_h^2 \ll \epsilon_h \leq \delta_h$ (see \cite{F-M-P, F-M-P2}). 

Another difference with respect to \cite{F-M-P2} is that, in general, one can not rely on a three-dimensional Korn's inequality on $\Omega$ to guarantee compactness of the sequence of cross-sectional displacements. However, the two-dimensional Korn's inequality proved in Theorem \ref{Korn} allows us to implicitly determine the cross-sectional displacements in the limit models through the characterization of~ $g$ (see the proof of Theorem \ref{deep}). 

\

The paper is organized as follows. In Section \ref{setting} we describe the setting of the problem. In Section \ref{preliminary} we recall some preliminary results and prove the rescaled Korn's inequality in curvilinear coordinates. In Section \ref{limit} we discuss some approximation results for displacements and bending moments. Section \ref{compactness} is devoted to the proof of the compactness results, while  Section \ref{characterization} to the liminf inequality. Finally, in Section \ref{construction} we construct the corresponding recovery sequences.\\
   
{\bf Notation.} We shall denote the canonical basis of $\mathbb{R}^3$ by $\{e_1,e_2,e_3\}$. If $\alpha:(0,L)\longrightarrow \mathbb{R}^m$ is a function of the $x_1$ variable, we shall denote its derivative, when it exists, by $\alpha'$, while if $\alpha:(0,1)\longrightarrow \mathbb{R}^m$ is a function of the $s$ variable, we shall denote its derivative by $\dot{\alpha}$. The $k$-th component of a vector $v$ will be denoted by $v_k$. For every $v,w\in \mathbb{R}^n$, we shall denote their scalar product by $v\cdot w$. We endow the space $\mathbb{M}^{n\times n}$ of $n\times n$ matrices with the euclidean norm $|M|:=\sqrt{Tr(M^T M)}=\sqrt{\sum_{i,j=1,\cdots,n}{m_{ij}^2}}$ and denote by the colon $:$ the associated scalar product.
For every $j\in\mathbb{N}$, $C^j_0(A;\mathbb{R}^m)$ and $C^{\infty}_0(A;\mathbb{R}^m)$ will be respectively the standard spaces of $C^j$ and $C^{\infty}$ functions with compact support in $A$.
\section{Setting of the problem}
\label{setting}
\noindent Let $(h),(\delta_h)$ be two sequences of
positive numbers such that $h\ten 0$ and 
\begin{equation}\label{scaling}\lim_{h\ten 0}\frac{\delta_h}{h}= 0.\end{equation} We
shall consider a thin-walled elastic beam, whose reference configuration is given by the set 
\begin{eqnarray}
\nonumber\Omega_h:&\hspace{-0.3 cm}=&\hspace{-0.3 cm}\Big{\{}x_1e_1+h\gamma(s)+\delta_htn(s):x_1\in
(0,L),\,s\in(0,1),\,t\in\Big(-\frac{1}{2},\frac{1}{2}\Big)\Big{\}},
\end{eqnarray}
where $\gamma:[0,1]\longrightarrow \mathbb{R}^3$,
$\gamma(s)=(0,\gamma_2(s),\gamma_3(s))$ is a simple, planar curve of class $C^6$
parametrized by arclength and $n(s)$ is the normal vector to the curve $\gamma$ at the point $\gamma(s)$. We shall denote by $\tau(s):=\dot{\gamma}(s)$ the tangent vector to
$\gamma$ at the point $\gamma(s)$, so that  
$$n(s)=\Bigg(\begin{array}{c}0\\ -\tau_3(s)\\\tau_2(s)\end{array}\Bigg)$$ for every $s\in [0,1]$. We define also the map $R_0:[0,1]\longrightarrow SO(3)$ given by
$$R_0(s):=\Big(e_1\,\Big |\,\tau(s)\, \Big |\, n(s)\Big)
$$ for every $s\in [0,1]$. For the sake of notation we introduce the two-dimensional vectors
$$\lin{\tau}(s):=\Big(\begin{array}{c}\tau_2(s)\\\tau_3(s)\end{array}\Big),\quad \lin{n}(s):=\Big(\begin{array}{c}-\tau_3(s)\\\tau_2(s)\end{array}\Big)$$ and the $2\times 2$ rotation
$$\lin{R}_0(s):=(\lin{\tau}(s)\Big|\lin{n}(s))$$
for every $s\in [0,1]$. Let \mbox{$k(s):=\dot{\tau}(s)\cdot n(s)$} be the curvature of $\gamma$ at the point $\gamma(s)$. We shall assume that $k$ is not identically equal to zero. Finally, let \mbox{$N,T:[0,1]\longrightarrow \mathbb{R}$} be the functions defined by \mbox{$N:=\gamma \cdot n$} and \mbox{$T:=\gamma \cdot \tau$}.\\

 For every $u\in W^{1,2}(\Omega_h;\mathbb{R}^3)$, we define the
elastic energy (per unit cross-section) associated with $u$ by 
\begin{equation}
 \cal E^h(u):=\frac{1}{h\delta_h} \int_{\Omega_h}{ W(\nabla u(x))dx}.
\end{equation}
The stored-energy density $W:\mathbb{M}^{3\times 3}\to[0,+\infty]$ is assumed to satisfy
the following conditions:
\begin{itemize}
\medskip
\item[(H1)] $W$ is continuous;
\medskip
\item[(H2)] $W(RF)=W(F)$ for every $R\in SO(3)$, $F\in\mthree$
(frame indifference);
\medskip
\item[(H3)] $W=0$ on $SO(3)$;
\medskip
\item[(H4)] $\exists C>0$ such that $W(F)\geq C\, \dist^2(F,SO(3))$ for every $F\in\mthree$;
\medskip
\item[(H5)] $W$ is of class $C^2$ in a neighbourhood of $SO(3)$,
\medskip
\end{itemize}
where $SO(3):=\{R\in\mthree:R^TR=Id,\ \det R=1\}$.

As usual in dimension reduction problems, we scale the deformations and the corresponding energy to a fixed domain. We set \mbox{$\Omega:=(0,L)\times (0,1)\times (-\frac{1}{2}, \frac{1}{2})$}. In the following we shall also consider the set
$$\omega:=(0,L)\times(0,1)$$
and the scaled cross-section
$$S:=(0,1)\times \big(-\frac{1}{2},\frac{1}{2}\big).$$ 
We define the maps $\psi^h:\Omega\longrightarrow
\Omega_h$ as
$$\psi^h(x_1,s,t):=x_1e_1+h\gamma(s)+\delta_h t n(s),$$
for every $(x_1,s,t)\in\Omega$ and we notice that there exists $h_0>0$ such that $\psi^h$ is a bijection for every $h\in(0,h_0)$. To every deformation $u\in W^{1,2}(\Omega_h;\mathbb{R}^3)$ we associate a scaled deformation $y\in W^{1,2}(\Omega;\mathbb{R}^3)$, defined by $y:=u\comp
\psi^h$, so that we can rewrite the elastic energy as 
\begin{equation}
\label{defjh}
 \cal{E}^h(u)=\cal{J}^h(y):=\myintom{\Big(\frac{h-\delta_h tk
}{h}\Big)W(\nabla_{h,\delta_h}y R_0^T)},
\end{equation}
where
$$\nabla_{h,\delta_h}y:=\Big(\partial_1
y\,\Big |\,\frac{1}{h-\delta_h t k}\,\partial_s
y\,\Big |\,\frac{1}{\delta_h}\,{\partial_t y}\Big).$$
We observe that $$\nabla_{h,\delta_h}\psi^h=R_0.$$
Moreover, since $k$ is a bounded function and \eqref{scaling} holds, we have that \begin{equation}\label{uncon}\frac{h-\delta_h t k}{h}\ten 1\end{equation} uniformly in $\lin{S}$. In particular, for $h$ small enough it follows that $h-\delta_h t k>0$
 for every $s\in[0,1]$ and \mbox{$t\in[-\frac{1}{2},\frac{1}{2}]$}.
 
  Throughout this article we shall consider sequences of scaled deformations $(y^h)$
in $W^{1,2}(\Omega;\mathbb{R}^3)$ satisfying 
\begin{equation}
\label{conditionenergy}
\myintom{\Big(\frac{h-\delta_h tk
}{h}\Big)W(\nabla_{h,\delta_h}y^h R_0^T)}\leq C\epsilon_h^2,
\end{equation}
 where $(\epsilon_h)$ is a given sequence of positive numbers. We shall mainly focus on the case where $(\epsilon_h)$  is infinitesimal of order larger or equal than $(\delta_h^2)$, that is, we shall assume that
 \begin{equation}
\label{l}
\exists \lim_{h\ten 0}\frac{\epsilon_h}{\delta_h^2}=:\ell\in [0,+\infty).
\end{equation}

\

A key role will be played by the quadratic form of linearized elasticity $Q_3:\mthree\longrightarrow [0,+\infty)$ defined by
$$Q_3(F):=D^2W(Id)F:F \quad \text{ for every }F\in \mthree.$$
The limiting functionals will involve the constant 
\begin{equation}
\label{defE}
\mathbb{E}:=\min_{a,b\in \mathbb{R}^3}{Q_3(e_1|a|b)}
\end{equation} and the quadratic form $Q_2:[0,1]\times \mathbb{R}^2\longrightarrow [0,+\infty)$ defined by
\begin{equation} 
\label{conrm}
Q_2(s,a,b)=\min_{\sigma_i\in \mathbb{R}}{Q_3\Bigg(R_0(s)\Bigg(
\begin{array}{ccc}
0&a&\sigma_1\\ a&b &\sigma_2\\\sigma_1&\sigma_2&\sigma_3
\end{array}
\Bigg)
R_0^T(s)\Bigg)}
\end{equation}
for any $s\in [0,1]$ and for any $(a,b)\in\mathbb{R}^2$. It is well known that, owing to (H2)--(H5), $Q_3$ is a positive semi-definite quadratic form and is positive definite on symmetric matrices. Hence, $\mathbb{E}>0$ and $Q_2(s,a,b)$ is strictly positive for every $s\in [0,1]$ and every $(a,b)\neq (0,0)$ .
\section{Preliminary results}
\label{preliminary}
\noindent In this section we collect some results which will be useful to prove a liminf inequality for the rescaled energies.

A first key tool to establish compactness of deformations with equibounded energies is the following rigidity estimate, due to Friesecke, James, and M\"uller \cite[Theorem 3.1]{F-J-M}.
\begin{teo}
\label{fjm}
Let $U$ be a bounded Lipschitz domain in $\mathbb{R}^n$, $n\geq 2$. Then there exists a constant $C(U)$ with the following properties: for every $v\in W^{1,2}(U;\mathbb{R}^n)$ there is an associated rotation $R\in SO(n)$ such that
$$\| \nabla v-R\|_{L^2(U)}\leq C(U)\|\dist(\nabla v, SO(n))\|_{L^2(U)}.$$
\end{teo}
\begin{oss}
\label{ofjm}
The constant $C(U)$ in Theorem \ref{fjm} is invariant by translations and dilations of $U$ and is uniform for families of sets which are uniform bi-Lipschitz images of a cube.
\end{oss}
Another crucial result in the proof of the liminf inequality is a modified version of the Korn's inequality in curvilinear coordinates. We refer to \cite{H} for a survey on Korn's inequality on bounded domains and to \cite{C2} for an overview on standard Korn's inequalities in curvilinear coordinates. 

We first fix some notation. We recall that $S=(0,1)\times (-\frac{1}{2},\frac{1}{2})$. For any $\epsilon>0$ and $v\in W^{1,2}(S;\mathbb{R}^2)$ we set
\begin{equation}
\label{eeps}
\lin{\nabla}_{\epsilon}v:=\Big(\frac{1}{1-\epsilon t k}\,{\partial_s v}\Big |\frac{1}{\epsilon}\,{\partial_t v}\Big)
\end{equation}
and we consider the subspace
$$M_{\epsilon}:=\Big\{v\in W^{1,2}(S;\mathbb{R}^2): \sym(\nn)=0\Big\}.$$
We remark that the expression $\sym(\nn)$ represents the linearized strain associated with the displacement $v\circ (\lin{\psi}{}^{\epsilon})^{-1}$, where
\begin{equation}
\label{psie}
\lin{\psi}{}^{\epsilon}(s,t):=\lin{\gamma}(s)+\epsilon t \lin{n}(s)
\end{equation}
for every $(s,t)\in S$. Since $M_{\epsilon}$ is closed in $W^{1,2}(S;\mathbb{R}^2)$, the orthogonal projection
$$\Pi_{\epsilon}:W^{1,2}(S;\mathbb{R}^2)\longrightarrow M_{\epsilon}$$
 is well defined.
We also introduce the set
\begin{equation}
\label{defm0}
M_0:=\Big\{v\in W^{1,2}(S;\mathbb{R}^2) : \partial_t v=0,\,\partial_s v\cdot \tg=0,\,\partial_s(\partial_s v\cdot \n)=0\Big\},
\end{equation} which will play a key role in the proof of the Korn's inequality. 

The following characterization of the spaces $M_{\epsilon}$ and $M_0$ can be given.
\begin{lem}
\label{struttmel}
Let $v\in M_0$. Then there exist $\alpha_1, \alpha_2,\alpha_3\in\mathbb{R}$ such that 
\begin{equation}
\label{struttm0}
v(s,t)=\Big(\begin{array}{c}\alpha_2\\\alpha_3\end{array}\Big)+\alpha_1\Big(\begin{array}{c}-\gamma_3(s)\\\gamma_2(s)\end{array}\Big)
\end{equation}
for every $(s,t)\in S$.\\
Let $v \in M_{\epsilon}$. Then there exist $\alpha_1, \alpha_2,\alpha_3 \in\mathbb{R}$ such that
\begin{equation}
\label{struttme}
v(s,t)=\Big(\begin{array}{c}\alpha_2\\\alpha_3\end{array}\Big)+\alpha_1\Big(\begin{array}{c}-\gamma_3(s)\\\gamma_2(s)\end{array}\Big)-\epsilon t \alpha_1\tg(s)
\end{equation}
for every $(s,t)\in S$.
\end{lem}
\begin{proof} 
It is immediate to see that, if $v\in M_0$, then $\partial_s v=\delta \lin{n}$ for some constant $\delta$, from which \eqref{struttm0} follows.

If $v\in M_{\epsilon}$, then $v\circ (\lin{\psi}{}^{\epsilon})^{-1}$ is an infinitesimal rigid displacement, that is, there exist $\alpha_1, \alpha_2,\alpha_3\in\mathbb{R}$ such that 
$$\Big(v\circ(\lin{\psi}^{\epsilon})^{-1}\Big)(x_2,x_3)=\Big(\begin{array}{c}\alpha_2\\\alpha_3\end{array}\Big)+\alpha_1\Big(\begin{array}{c}-x_3\\x_2\end{array}\Big)$$
for every $(x_2,x_3)\in \lin{\psi}{}^{\epsilon}(S)$. This implies \eqref{struttme}.
\end{proof}
Finally, we recall a lemma which is due to J.L. Lions.
\begin{lem}[Lemma of J.L. Lions]
\label{lio}
Let $U$ be a bounded, connected, open set in $\mathbb{R}^n$ with Lipschitz boundary and let v be a distribution on $U$. If $v\in W^{-1,2}(U)$ and $\partial_i v \in W^{-1,2}(U)$ for $i=1,\cdots,n$, then $v\in L^2(U)$. 
\end{lem}We refer to \cite[Section 1.7]{C2} for a detailed bibliography on this lemma. 

We are now in a position to state and prove a rescaled Korn's inequality in curvilinear coordinates.
\begin{teo}[Korn's inequality]
\label{Korn}
 There exist two constants $\epsilon_0>0$ and $C>0$ such that for any $\epsilon\in(0,\epsilon_0)$, $v\in W^{1,2}(S;\mathbb{R}^2)$, there holds
\begin{equation}
\label{korn}
\|v-\Pi_{\epsilon}(v)\|_{W^{1,2}(S)}\leq \frac{C}{\epsilon}\|\sym(\nn)\|_{L^2(S)}.\end{equation}
\end{teo}
\begin{oss}
An analogous dependance of Korn's constant on the thickness of a thin structure has been proved, e.g. in \cite[Proposition 4.1]{K-V}, in the case of a thin plate with rapidly varying thickness.
\end{oss}
\begin{proof}[Proof of Theorem \ref{Korn}] By contradiction, assume there exist a sequence $(\epsilon_j)$ and a sequence of maps $(v^j)\subset W^{1,2}(S;\mathbb{R}^2)$ such that $\epsilon_j \ten 0$ and 
\begin{equation}
\label{contrk}
\|v^j-\Pi_{\epsilon_j}(v^j)\|_{W^{1,2}(S)}>\frac{j}{\epsilon_j}\|\sym(\nnn)\|_{L^2(S)},
\end{equation}
for every $j\in\mathbb{N}$.
Up to normalizations, we can assume that \mbox{$\|v^j-\Pi_{\epsilon_j}(v^j)\|_{W^{1,2}(S)}=1.$} We also set \mbox{$\phi^j:=v^j-\Pi_{\epsilon_j}(v^j)$.}
By definition $\phi^j\in W^{1,2}(S;\mathbb{R}^2)$,  $\phi^j$ is orthogonal to $M_{\epsilon_j}$  in the sense of $W^{1,2}$, and \begin{equation}\label{estcurvgrad}\|\sym(\nnp)\|_{L^2(S)}<\frac{\epsilon_j}{j}\end{equation} for every $j$. Since $\|\phi^j\|_{W^{1,2}}=1$ for every $j$, there exists $\phi\in W^{1,2}(S;\mathbb{R}^2)$ such that, up to subsequences, $\phi^j \deb \phi$ in $W^{1,2}(S;\mathbb{R}^2)$. 

Let now $u\in M_0$. Then, there exists a sequence $(u^j)$ such that $u^j \in M_{\epsilon_j}$ for all $j \in \mathbb{N}$ and $u^j \ten u$ in $W^{1,2}(S;\mathbb{R}^2)$. Indeed, from Lemma \ref{struttmel}, it follows that $$u=\Big(\begin{array}{c}\alpha_2\\\alpha_3\end{array}\Big)+\alpha_1\Big(\begin{array}{c}-\gamma_3\\\gamma_2\end{array}\Big)$$ for some $\alpha_1,\alpha_2,\alpha_3\in\mathbb{R}$. The maps $u^j$ given by 
$$u^j:=u-\epsilon_j t \delta \tg $$ have the required properties. Since $\scal{\phi^j}{u^j}_{W^{1,2}}=0$ for any $j\in \mathbb{N}$, passing to the limit, we obtain that $\scal{\phi}{u}_{W^{1,2}}=0$ for every $u\in M_0$; hence $\phi$ is orthogonal to $M_0$ in the sense of $W^{1,2}$. 

To deduce a contradiction we shall prove that the convergence of $(\phi^j)$ is actually strong in $W^{1,2}(S;\mathbb{R}^2)$ and $\phi \in M_0$.\\
Indeed, since from \eqref{estcurvgrad} 
\begin{equation}
\label{coneen}
\sym(\nnp)_{11}\ten 0,\quad \epsilon_j\sym(\nnp)_{12}\ten 0\quad \text{ and }\quad
 \sym(\nnp)_{22}\ten 0 
 \end{equation} strongly in $L^2(S)$, it is immediate to see that 
\begin{equation}
\label{strcon}
      \partial_s \phi^j \cdot \tg \ten 0, \quad \partial_t \phi^j\cdot \tg\ten 0, \quad \text{ and } \frac{1}{\epsilon_j}\partial_t \phi^j\cdot \n \ten 0
\end{equation}
 strongly in $L^2(S)$. To show the strong convergence of $\phi^j$ in $W^{1,2}(S;\mathbb{R}^2)$, it remains to prove that $\partial_s \phi^j \cdot \n \ten \partial_s \phi \cdot \n$ strongly in $L^2(S)$. By Lemma  \ref{lio} and by the closed graph Theorem, it is enough to prove that 
$$\partial_s \phi^j \cdot \n \ten \partial_s \phi \cdot \n\quad\text{ and }\quad
\nabla(\partial_s \phi^j \cdot \n)\ten \nabla(\partial_s \phi \cdot \n)$$ strongly in $W^{-1,2}.$ Convergence of $(\partial_s \phi^j \cdot \n)$ is immediate as $(\phi^j)$ is strongly converging in $L^2(S;\mathbb{R}^2)$. Strong convergence in $W^{-1,2}(S;\mathbb{R}^2)$ of $(\partial_t\partial_s \phi^j \cdot \n)$ follows from the identity $$\partial_t\partial_s \phi^j \cdot \n=\partial_s(\partial_t \phi^j \cdot \n)+k\partial_t \phi^j \cdot \tg$$
and  from \eqref{strcon}. To prove convergence of $(\partial_s (\partial_s \phi^j \cdot \n))$ we notice that, by \eqref{estcurvgrad},
\begin{equation}
\label{pdte}
\frac{1}{\epsilon_j}\|\partial_t(\sym(\nnp)_{11})\|_{W^{-1,2}(S)}\leq \frac{1}{j}
\end{equation}
for all $j \in \mathbb{N}$ and by \eqref{coneen}, \mbox{$\partial_s(\sym(\nnp)_{12})\ten 0$} strongly in $W^{-1,2}(S)$.
Furthermore,
\begin{eqnarray}
\nonumber \frac{1}{\epsilon_j}\partial_t (\sym(\nnp)_{11})&=&\frac{\partial_t\partial_s \phi^j \cdot \tg}{\epsilon_j(1-\epsilon_j t k)}+\frac{k(\partial_s \phi^j\cdot \tg)}{1-\epsilon_j t k}\\
\nonumber &\hspace{-3.5cm}=&\hspace{-2cm}\frac{2\partial_s (\sym(\nnp)_{12})}{1-\epsilon_j t k}-\frac{k}{1-\epsilon_jtk}\frac{\partial_t \phi^j \cdot \n}{\epsilon_j}\\
\nonumber &\hspace{-3.5cm}-&\hspace{-2cm}\frac{1}{1-\epsilon_j t k}\partial_s\Big(\frac{\partial_s \phi^j \cdot \n}{1-\epsilon_j t k}\Big)
+\frac{k(\partial_s \phi^j\cdot \tg)}{1-\epsilon_j t k}.
\end{eqnarray}
Then, using \eqref{strcon} and \eqref{pdte}, we deduce \begin{equation}\label{fincondm0}\partial_s(\partial_s \phi^j \cdot \n)\ten 0\text{ strongly in }W^{-1,2}(S).\end {equation}

It follows that $\phi^j\ten \phi$ strongly in $W^{1,2}(S;\mathbb{R}^2)$ and, since $\|\phi^j\|_{W^{1,2}}=1$ for any $j \in \mathbb{N}$, also $\|\phi\|_{W^{1,2}}=1.$ On the other hand, by \eqref{strcon} and \eqref{fincondm0}, $\phi \in M_0$. Since $\phi$ is orthogonal to $M_0$ in the sense of $W^{1,2}$, then $\phi$ must be identically equal to zero. This gives a contradiction and completes the proof. 
\end{proof}
 We conclude this section by proving a technical lemma. We recall that $\omega=(0,L)\times(0,1).$
\begin{lem}
\label{conv2e}
Let $(\alpha^h_i)\subset W^{-2,2}(0,L)$, $i=1,2,3,$ and let $f\in W^{-2,2}(\omega)$ be such that
\begin{equation}
\label{hpah}
\alpha^h_1 N+\alpha_2^h \tau_2+\alpha^h_3\tau_3 \deb f
\end{equation}
weakly in $W^{-2,2}(\omega)$, as $h\ten 0$. Then, there exist $\alpha_i\in W^{-2,2}(0,L)$, $i=1,2,3$, such that for every $i$
\begin{equation}
\label{cvah}
\alpha^h_i\deb \alpha_i \end{equation}
weakly in $W^{-2,2}(0,L)$, as $h\ten 0$, and
\begin{equation}
\label{rapf}
f=\alpha_1 N+\alpha_2 \tau_2+\alpha_3 \tau_3.
\end{equation}
If, in addition, there exists $g\in L^2(\omega)$ such that $f=\partial_s g$, then $\alpha_i\in L^2(0,L)$ for any $i=1,2,3$. If $f=0$, then $\alpha_i=0$ for any $i=1,2,3$.
\end{lem}
\begin{proof}
For the sake of notation, throughout the proof we use the symbol $\scal{\cdot \,}{\cdot}$ to denote the duality pairing between $W^{-2,2}(\omega)$ and $W^{2,2}_0(\omega)$.

We recall that any $\alpha\in W^{-2,2}(0,L)$ can be identified with an element of the space \mbox{$W^{-2,2}(\omega)$} by setting
\begin{equation}
\label{pass00}
\scal{\alpha}{\delta}:=\myintt{\scal{\alpha}{\delta(s,\cdot)}_{W^{-2,2}(0,L),W^{2,2}_0(0,L)}}
\end{equation}  
for any $\delta\in C^{\infty}_0(\omega)$ and extending it by density to  $W^{2,2}_0(\omega)$. Moreover,
for any $\alpha\in W^{-2,2}(0,L)$ and $\beta\in C(0,1)$, we define the product $\alpha\beta$ as
 $$\scal{\alpha\beta}{\delta}:=\scal{\alpha}{\beta\delta}=\myintt{\scal{\alpha}{\delta(s,\cdot)}_{W^{-2,2}(0,L),W^{2,2}_0(0,L)}\beta(s)}$$
 for every $\delta\in C^{\infty}_0(\omega)$.
 
 Let now $\varphi\in W^{2,2}_0(0,L)$ and $\psi\in C^{j+2}_0(0,1)$, with $j\in\mathbb{N}$.
We claim that
\begin{equation}
\label{claima}
\scal{\alpha^h_i}{\varphi\partial_s^j\psi}=0.
\end{equation}
Indeed, let $(\varphi^l)\subset C^{\infty}_0(0,L)$ be such that $\varphi^l\ten \varphi$ in $W^{2,2}(0,L)$. Then, 
$$\scal{\alpha^h_i}{\varphi\partial_s^j\psi}=\lim_{l\ten+\infty}\scal{\alpha^h_i}{\varphi^l\partial_s^j\psi}.$$
On the other hand, 
$$\scal{\alpha^h_i}{\varphi^l\partial_s^j\psi}=\myintt{\partial_s\scal{\alpha^h_i}{\varphi^l \partial_s^{j-1}\psi}_{W^{-2,2}(0,L),W^{2,2}_0(0,L)}}=0$$
for any $l\in\mathbb{N}$. Therefore, claim \eqref{claima} is proved.

By \eqref{hpah}, for any $\varphi\in W^{2,2}(0,L)$, $\psi\in C^{j+2}_0(0,1)$, we have
$$\scal{\alpha^h_1 N}{\varphi\partial_s^j \psi}+\sum_{i=2,3}{\scal{\alpha^h_i\tau_i}{\varphi\partial_s^j\psi}}\ten \scal{f}{\varphi\partial_s^j\psi}.$$
Claim \eqref{claima} yields then
\begin{equation}
\label{pass}
\scal{\alpha^h_1 kT+\alpha^h_2 k\tau_3-\alpha^h_3 k\tau_2}{\varphi\partial_s^{j-1}\psi}\ten \scal{f}{\varphi\partial_s^j\psi}.
\end{equation}
Hence, choosing $j=1$ we obtain
\begin{equation}
\label{E2}
\scal{\alpha^h_1 kT+\alpha^h_2 k\tau_3-\alpha^h_3 k\tau_2}{\varphi\psi}\ten \scal{f}{\varphi\partial_s\psi}
\end{equation}
for any $\varphi\in W^{2,2}_0(0,L)$ and $\psi\in C^{3}_0(0,1)$.

Let now $\varphi\in W^{2,2}_0(0,L)$, $\psi\in C^{j+3}_0(0,1)$. Taking $\varphi\partial_s^j \psi$ as test function in \eqref{E2} and applying again \eqref{claima}, we deduce
\begin{equation}
\nonumber
\scal{-\alpha^h_1(\dot{k}T+k+k^2N)-\alpha^h_2(\dot{k}\tau_3+k^2\tau_2)+\alpha^h_3(\dot{k}\tau_2-k^2\tau_3)}{\varphi\partial_s^{j-1}\psi}\ten\scal{f}{\varphi\partial_s^{j+1}\psi},
\end{equation} 
which in turn gives
\begin{equation}
\label{E3}
\scal{-\alpha^h_1(\dot{k}T+k+k^2N)-\alpha^h_2(\dot{k}\tau_3+k^2\tau_2)+\alpha^h_3(\dot{k}\tau_2-k^2\tau_3)}{\varphi\psi}\ten\scal{f}{\varphi\partial_s^{2}\psi},
\end{equation} 
for any $\varphi\in W^{2,2}(0,L)$, $\psi\in C^{4}_0(0,1)$.

Consider $\phi\in C^{\infty}_0(0,1)$. By regularity of the curve $\gamma$, the map $k\phi\in C^4_0(0,1)$. Therefore, for any $\varphi\in W^{2,2}_0(0,L)$ we can choose $\varphi k\phi$ as test function in \eqref{E3} obtaining
$$\scal{-\alpha^h_1(k\dot{k}T+k^2+k^3N)-\alpha^h_2(k\dot{k}\tau_3+k^3\tau_2)+\alpha^h_3(k\dot{k}\tau_2-k^3\tau_3)}{\varphi\phi}\ten\scal{f}{\varphi\partial_s^{2}(k\phi)}.$$
On the other hand, by \eqref{hpah}
$$\scal{\alpha^h_1N+\sum_{1=2,3}{\alpha^h_i\tau_i}}{\varphi k^3\phi}\ten\scal{f}{\varphi k^3\phi},$$
and by \eqref{E2}
$$\scal{\alpha^h_1 kT+\alpha^h_2 k\tau_3-\alpha^h_3 k\tau_2}{\varphi\dot{k}\phi}\ten\scal{f}{\varphi\partial_s(\dot{k}\phi)}.$$
Collecting the previous remark we deduce
\begin{equation}
\label{pass1}
\scal{\alpha^h_1}{\varphi k^2\phi}\ten\scal{f}{\varphi(-\partial_s^2(k\phi)-k^3\phi-\partial_s(\dot{k}\phi))}
\end{equation}
for any $\varphi\in W^{2,2}(0,L)$, $\phi\in C^{\infty}_0(0,1)$.

Let now $\lin{\phi}\in C^{\infty}_0(0,1)$ be such that $\myintt{k^2\lin{\phi}}=1$ (such $\lin{\phi}$ exists because $k$ is not identically equal to zero in $(0,1)$). Convergence \eqref{pass1} implies that 
\begin{equation}
\label{pass25}
\alpha^h_1\deb\alpha_1\text{ weakly in }W^{-2,2}(0,L),
\end{equation}  
where
\begin{equation}
\label{ae1}
\scal{\alpha_1}{\varphi}_{W^{-2,2}(0,L),W^{2,2}_0(0,L)}=\scal{f}{\varphi(-\partial_s^2(k\lin{\phi})-k^3\lin{\phi}-\partial_s(\dot{k}\lin{\phi}))}
\end{equation}
for every $\varphi\in W^{2,2}_0(0,L)$. By definition \eqref{pass00} it is immediate to see that, identifying $\alpha^h_1,\alpha_1$ with elements of $W^{-2,2}(\st)$, we also have
\begin{equation}
\label{ae2}
\alpha^h_1\deb \alpha_1
\end{equation}
weakly in $W^{-2,2}(\st)$.

Let again $\varphi\in W^{2,2}_0(0,L)$, $\phi\in C^{\infty}_0(0,1)$. Taking $\varphi k \tau_2 \phi$ and $\varphi\tau_3\phi$ as test function respectively in $\eqref{hpah}$ and \eqref{E2} we deduce
\begin{equation}
\label{pass3}
\scal{\alpha^h_1 N+\alpha_2^h \tau_2+\alpha^h_3\tau_3}{\varphi k \tau_2\phi}\ten \scal{f}{\varphi k \tau_2\phi}
\end{equation}
and
\begin{equation}
\label{pass4}
\scal{\alpha^h_1 kT+\alpha^h_2 k\tau_3-\alpha^h_3 k\tau_2}{\varphi \tau_3\phi}\ten \scal{f}{\varphi\partial_s(\tau_3\phi)}.
\end{equation}
Summing \eqref{pass3} and \eqref{pass4} and using \eqref{ae2}, we obtain
$$\scal{\alpha^h_2}{k\phi \varphi}\ten \scal{f}{\varphi(k\tau_2\phi+\partial_s(\tau_3\phi))}-\scal{\alpha_1}{\varphi k \gamma_3\phi}$$
for any $\varphi\in W^{2,2}_0(0,L)$ and for any $\phi\in C^{\infty}_0(0,1)$.

Choosing $\capp{\phi}$ such that $\myintt{k\capp{\phi}}=1$ and arguing as in the proof of \eqref{pass25}, we deduce that 
\begin{equation}
\label{pass5}
\alpha^h_2\deb\alpha_2 \text{ weakly in }W^{-2,2}(0,L),
\end{equation}
where 
\begin{equation}
\label{ae3}
\scal{\alpha_2}{\varphi}_{W^{-2,2}(0,L),W^{2,2}(0,L)}= \scal{f}{\varphi(k\tau_2\capp{\phi}+\partial_s(\tau_3\capp{\phi}))}-\scal{\alpha_1}{\varphi k \gamma_3 \capp{\phi}}
\end{equation}
for any $\varphi\in W^{2,2}_0(0,L)$.

Similarly, one can prove that 
$$\alpha^h_3\deb\alpha_3 \text{ weakly in }W^{-2,2}(0,L)$$
where 
\begin{equation}
\label{ae4}
\scal{\alpha_3}{\varphi}_{W^{-2,2}(0,L), W^{2,2}_0(0,L)}=\scal{f}{\varphi( k\tau_3\capp{\phi}-\partial_s(\tau_2\capp{\phi}))}+\scal{\alpha_1}{\varphi k \gamma_2 \capp{\phi}}
\end{equation}
for any $\varphi\in W^{2,2}_0(0,L)$.

Combining \eqref{hpah}, \eqref{ae2}, \eqref{pass5}, and \eqref{ae4}, we obtain the representation \eqref{rapf}.

If $f=\partial_s g$, with $g\in L^2(\omega)$, then by \eqref{ae1}
$$\scal{\alpha_1}{\varphi}_{W^{-2,2}(0,L), W^{2,2}_0(0,L)}=\myintss{g\partial_s(\partial_s^2(k\lin{\phi})+k^3\lin{\phi}+\partial_s(\dot{k}\lin{\phi}))\varphi},$$
for any $\varphi\in W^{2,2}_0(0,L)$. This implies that $\alpha_1\in L^2(0,L)$. Similarly
equalities \eqref{ae3} and \eqref{ae4} yield $\alpha_2,\alpha_3\in L^2(0,L)$. 

Finally, if $f=0$, by \eqref{ae1}, \eqref{ae3} and \eqref{ae4} we deduce immediately that $\alpha_i=0$ for every $i$.\end{proof}
\section{Limit classes of displacements and bending moments and approximation results}
\label{limit}
\noindent In this section we introduce some classes of displacements and bending moments, that will emerge in the limit models, and we discuss their properties and their approximation by means of smooth functions.  

We begin by introducing the limit class of the tangential derivatives of the tangential displacements
\begin{multline}
\cal{G}:=\Big\{g\in L^2(\omega): \exists (v^{\epsilon})\subset C^5(\lin{\omega};\mathbb{R}^3) \text{ such that }
\partial_s v_1^{\epsilon}+\partial_1 v^{\epsilon} \cdot \tau=0,\\ \label{classg} \partial_s v^{\epsilon} \cdot \tau=0\text{ for every }\epsilon>0 \text{ and }  g=\lim_{\epsilon \ten 0}{\partial_1 v_1^{\epsilon}}\Big\},
\end{multline}
where the limit is intended with respect to the strong convergence of $L^2(\omega)$. 
In other words, if for every $v\in W^{1,2}(\omega;\mathbb{R}^3)$ we consider the symmetric gradient $e(v)\in L^2(\omega;\mathbb{M}^{2\times 2}_{\sym})$ of $v$, defined by
\begin{equation}
\label{syg}
e(v):=\left(\begin{array}{cc}\partial_1 v_1&\frac{1}{2}(\partial_s v_1+\partial_1 v\cdot \tau)\\\frac{1}{2}(\partial_s v_1+\partial_1 v\cdot \tau)&\partial_s v\cdot \tau\end{array}\right),
\end{equation}
a function $g\in L^2(\omega)$ belongs to $\cal{G}$ if and only if there exists a sequence $(v^{\epsilon})\subset C^5(\lin{\omega};\mathbb{R}^3)$ such that
$$e(v^{\epsilon})=\left(\begin{array}{cc}\partial_1 v^{\epsilon}_1&0\\0&0\end{array}\right)\ten \left(\begin{array}{cc}g&0\\0&0\end{array}\right)$$
strongly in $L^2(\omega;\mathbb{M}^{2\times 2}_{\sym})$ as $\epsilon \ten 0$.
\begin{lem}
\label{scriverelemma}
Let $g\in L^2(\omega)$ and assume there exists a sequence $(v^{\epsilon})\subset W^{1,2}(\omega;\mathbb{R}^3)$ such that
\begin{equation}
\label{cevh}
e(v^{\epsilon})\deb \Big(\begin{array}{cc}g&0\\0&0\end{array}\Big)
\end{equation}
weakly in $L^2(\omega;\mathbb{M}^{2\times 2}_{\sym})$ as $\epsilon \ten 0$. Then $g\in \cal{G}$.
\end{lem}
\begin{proof}
Condition \eqref{cevh} can be rewritten as
\begin{eqnarray}
 \label{ec1}&\partial_1 v^{\epsilon}_1\deb g &\text{ weakly in }L^2(\omega),\\
\label{ec2}&\partial_s v^{\epsilon}_1+\partial_1 v^{\epsilon}\cdot \tau\deb 0 &\text{ weakly in }L^2(\omega),\\
\label{ec3}&\partial_s v^{\epsilon}\cdot \tau\deb 0&\text{ weakly in }L^2(\omega).
\end{eqnarray}
By Mazur's Lemma, we may assume that the convergence in \eqref{ec1}, \eqref{ec2} and \eqref{ec3} is strong in $L^2(\omega_{\delta})$. For every $\epsilon$, let $\til{u}^{\epsilon}\in W^{1,2}(\omega)$, with $\partial_1^2 \til{u}^{\epsilon}\in L^2(\omega)$, be such that $\partial_1 \til{u}^{\epsilon}=v^{\epsilon}_1$. By \eqref{ec2} and by Poincar\'e inequality 
\begin{equation}
\label{cevh2}
\partial_s \til{u}^{\epsilon}+v^{\epsilon}\cdot \tau-\myfintol{\partial_s \til{u}^{\epsilon}}-\myfintol{v^{\epsilon}\cdot \tau}\ten 0
\end{equation}
strongly in $L^2(\omega)$. Let now $\nu^{\epsilon}\in W^{1,2}(\omega)$ be such that $\partial_s \nu^{\epsilon}=v^{\epsilon}\cdot \tau$. Setting
$$u^{\epsilon}:=\til{u}^{\epsilon}-\myfintol{\til{u}^{\epsilon}}-\myfintol{\nu^{\epsilon}},$$
then $u^{\epsilon}\in W^{1,2}(\omega)$ with $\partial_1^2 u^{\epsilon}\in L^2(\omega)$ for every $\epsilon$, \eqref{cevh2} yields
\begin{equation}
\label{ec2bis}
\partial_s u^{\epsilon}+v^{\epsilon}\cdot \tau\ten 0
\end{equation}
strongly in $L^2(\omega)$ and by \eqref{ec1} there holds
\begin{equation}
\label{ec1bis}
\partial_1^2 u^{\epsilon}\ten g
\end{equation}
strongly in $L^2(\omega)$. 

We want to approximate ${u}^{\epsilon}$ and ${v}^{\epsilon}$ by smooth functions in such a way that \eqref{ec1bis} holds and the quantities in \eqref{ec3} and \eqref{ec2bis} are equal to zero for every $\epsilon>0$. To this purpose, we first extend ${u}^{\epsilon}$ and ${v}^{\epsilon}$ to the set $$\omega_{\delta}:=(-\delta,L+\delta)\times(0,1),$$ 
with $0<\delta<\frac{L}{3}$. For every $\epsilon$, we define
 \begin{equation}
 \nonumber
 \mathring{v}^{\epsilon}(x_1,s):=\begin{cases}
 {v}^{\epsilon}(x_1,s) &\text{ in }\omega,\\
 6{v}^{\epsilon}(-x_1,s)-8{v}^{\epsilon}(-2x_1,s)+3v^{\epsilon}(-3x_1,s) &\text{ in }(-\delta,0)\times(0,1),\\
 6{v}^{\epsilon}(2L-x_1,s)-8{v}^{\epsilon}(3L-2x_1,s)+3v^{\epsilon}(4L-3x_1,s)&\text{ in }(L,L+\delta)\times(0,1)\end{cases}
 \end{equation}
 and
 \begin{equation}
 \nonumber
 \mathring{u}^{\epsilon}(x_1,s):=\begin{cases}
 {u}^{\epsilon}(x_1,s) &\text{ in }\omega,\\
 6{u}^{\epsilon}(-x_1,s)-8{u}^{\epsilon}(-2x_1,s)+3u^{\epsilon}(-3x_1,s)&\text{ in }(-\delta,0)\times(0,1),\\
 6{u}^{\epsilon}(2L-x_1,s)-8{u}^{\epsilon}(3L-2x_1,s)+3u^{\epsilon}(4L-3x_1,s)&\text{ in }(L,L+\delta)\times(0,1).\end{cases}
 \end{equation}
  Clearly, $\mathring{v}^{\epsilon}$ and $\mathring{u}^{\epsilon}$ are extensions of ${v}^{\epsilon}$ and ${u}^{\epsilon}$, respectively, to $\omega_{\delta}$.
  Moreover, $\mathring{u}^{\epsilon}\in W^{1,2}(\omega_{\delta})$ with $\partial_1^2\mathring{u}^{\epsilon}\in L^2(\omega_{\delta})$, $\mathring{v}^{\epsilon}\in W^{1,2}(\omega_{\delta})$ and both \eqref{ec3} and \eqref{ec2bis} still hold in $\omega_{\delta}$.\\
   Furthermore, defining
    \begin{equation}
  \nonumber
  \mathring{g}:=\begin{cases}
{g}(x_1,s) &\text{ in }\omega,\\
 6{g}(-x_1,s)-32{g}(-2x_1,s)+27g(-3x_1,s)&\text{ in }(-\delta,0)\times(0,1),\\
 6{g}(2L-x_1,s)-32{g}(3L-2x_1,s)+27g(4L-3x_1,s)&\text{ in }(L,L+\delta)\times(0,1)\end{cases}
  \end{equation} 
  we have that $\mathring{g}\in L^2(\omega_{\delta})$, $\mathring{g}=g$ a.e. in $\omega$, and 
  \begin{equation}
\label{gc2}
 \partial_1^2\mathring{u}^{\epsilon}\ten \mathring{g}
 \end{equation}
   strongly in $L^2(\omega_{\delta})$. 
   
  We set $\mathring{v}^{\epsilon}_t:=\mathring{v}^{\epsilon}\cdot n$ and $\mathring{v}^{\epsilon}_s:=\mathring{v}^{\epsilon}\cdot \tau$. For every $\epsilon$, let $\lin{v}^{\epsilon}_t\in C^{\infty}(\lin{\omega}_{\delta})$ be such that 
 \begin{equation}
 \label{gd}
 \|\lin{v}^{\epsilon}_t-\mathring{v}^{\epsilon}_t\|_{W^{1,2}(\omega_{\delta})}\leq {C\epsilon}.
 \end{equation}
 Let now $\lin{v}^{\epsilon}_s\in C^5(\lin{\omega}_{\delta})$ be the solution of
 \begin{equation}
 \label{ge}
 \partial_s \lin{v}^{\epsilon}_s=k\lin{v}^{\epsilon}_t \quad \text{ in }\omega_{\delta},
 \end{equation}
 satisfying $\int_{0}^1{\lin{v}^{\epsilon}_s(x_1,s)ds}\in C^{\infty}(-\delta, L+\delta)$, with
 $$\int_{0}^1{\lin{v}^{\epsilon}_s(x_1,s)ds}- \int_0^1{\mathring{v}^{\epsilon}_s(x_1,s)ds}\ten 0$$
 strongly in $L^2(-\delta, L+\delta)$.
 By \eqref{ge} we deduce
 $$\|\partial_s(\lin{v}^{\epsilon}_s-\mathring{v}^{\epsilon}_s)\|_{L^2(\omega_{\delta})}\leq \|k(\lin{v}^{\epsilon}_t-\mathring{v}^{\epsilon}_t)\|_{L^2(\omega_{\delta})}+\|k\mathring{v}^{\epsilon}_t-\partial_s \mathring{v}^{\epsilon}_s\|_{L^2(\omega_{\delta})}.$$
 Hence, owing to \eqref{ec3} and \eqref{gd},
 \begin{equation}
 \label{gf}
 \|\partial_s(\lin{v}^{\epsilon}_s-\mathring{v}^{\epsilon}_s)\|_{L^2(\omega_{\delta})}\ten 0. 
 \end{equation}
 By Poincar\'e inequality we deduce 
 \begin{equation}
 \label{gf1}
 \|\lin{v}^{\epsilon}_s-\mathring{v}^{\epsilon}_s\|_{L^2(\omega_{\delta})}\ten 0.
 \end{equation}
 Finally, let $\lin{u}^{\epsilon}\in C^6(\lin{\omega}_{\delta})$ be such that
\begin{equation}
\label{gg}
\partial_s \lin{u}^{\epsilon}+\lin{v}^{\epsilon}_s=0\quad \text{ in }\omega_{\delta},
\end{equation}
with $\myintt{\lin{u}^{\epsilon}(x_1,s)}\in C^{\infty}(-\delta,L+\delta)$
and $$\myintt{\lin{u}^{\epsilon}(x_1,s)}-\myintt{\mathring{u}^{\epsilon}(x_1,s)}\ten 0$$ strongly in $L^2(-\delta, L+\delta)$. By \eqref{gg} we have
$$\|\partial_s(\lin{u}^{\epsilon}-\mathring{u}^{\epsilon})\|_{L^2(\omega_{\delta})}\leq \|\partial_s\mathring{u}^{\epsilon}+\mathring{v}^{\epsilon}_s\|_{L^2(\omega_{\delta})}+\|\mathring{v}^{\epsilon}_s-\lin{v}^{\epsilon}_s\|_{L^2(\omega_{\delta})}.$$
Therefore, by \eqref{ec2bis} and \eqref{gf1}, $$\partial_s (\lin{u}^{\epsilon}-\mathring{u}^{\epsilon})\ten 0$$ strongly in $L^2(\omega_{\delta})$. Hence, by Poincar\'e inequality 
\begin{equation}
\label{gh}
\lin{u}^{\epsilon}-\mathring{u}^{\epsilon}\ten 0
\end{equation} strongly in $L^2(\omega_{\delta})$.
 
 To have convergence of the second derivative in the $x_1$ variable of the sequence $(\lin{u}^{\epsilon})$, we regularize both $(\lin{u}^{\epsilon})$ and $(\lin{v}^{\epsilon})$ by mollification in the $x_1$ variable. Let $\rho\in C^{\infty}_0(-\lambda,\lambda)$ with $0<\lambda<\delta$. Defining
 \begin{numcases}{}
 \nonumber \capp{v}^{\epsilon}(x_1,s):=(\lin{v}^{\epsilon}(\cdot,s)\ast\rho)(x_1),&\\
 \nonumber \capp{u}^{\epsilon}(x_1,s):=(\lin{u}^{\epsilon}(\cdot,s)\ast\rho)(x_1),&
 \end{numcases}
 for a.e. $(x_1,s)\in \omega$ and for every $\epsilon>0$, we have $(\capp{v}^{\epsilon}_t)\subset C^{\infty}(\lin{\omega})$, $(\capp{v}^{\epsilon}_s)\subset C^5(\lin{\omega})$, and $(\capp{u}^{\epsilon})\subset C^6(\lin{\omega})$.
 By \eqref{ge} and \eqref{gg}, we deduce
 \begin{equation}
 \nonumber \partial_s \capp{v}^{\epsilon}_s=k\capp{v}^{\epsilon}_t\quad \text{ and }\quad\partial_s \capp{u}^{\epsilon}+\capp{v}^{\epsilon}_s=0.
 \end{equation}
 
 \noindent Moreover, by \eqref{gh},
 $$\partial_1^2(\capp{u}^{\epsilon}- (\mathring{u}^{\epsilon}(\cdot,s)\ast \rho))=(\lin{u}^{\epsilon}(\cdot,s)-\mathring{u}^{\epsilon}(\cdot,s))\ast\rho''\ten 0$$ strongly in $L^2(\omega)$ as $\epsilon\ten 0$. On the other hand, by \eqref{gc2},
 $$\partial_1^2(\mathring{u}^{\epsilon}(\cdot,s)\ast\rho)=\partial_1^2 \mathring{u}^{\epsilon}(\cdot,s)\ast \rho \ten \mathring{g}(\cdot,s)\ast\rho$$
 strongly in $L^2(\omega)$ as $\epsilon\ten 0$; hence $$\partial_1^2 \capp{u}^{\epsilon}\ten \mathring{g}(\cdot,s)\ast\rho$$ strongly in $L^2(\omega)$ as $\epsilon\ten 0$.
 
 The conclusion of the lemma follows considering a sequence of convolution kernels in the $x_1$ variable, and applying a diagonal argument.
\end{proof}

\begin{oss}
\label{recuse}
An equivalent characterization of $\cal{G}$ is the following:
\begin{multline}
\cal{G}=\Big\{g\in L^2(\omega):\exists (u^{\epsilon})\subset C^5(\lin{\omega}),(z^{\epsilon})\subset C^4(\lin{\omega})\text{ such that }\\
\label{classg1} \partial_s^2 u^{\epsilon}=kz^{\epsilon} \text{ for every }\epsilon>0\text{ and }
  g=\lim_{\epsilon\ten 0}\partial_1 u^{\epsilon}\Big\},
\end{multline}
where the limit is intended with respect to the strong convergence in $L^2(\omega)$.

Indeed, let $\cal{G}'$ be the class defined in the right-hand side of \eqref{classg1}. If $g\in\cal{G}$, setting $u^{\epsilon}=v^{\epsilon}_1$ and $z^{\epsilon}=-\partial_1 v^{\epsilon}\cdot n$ for every $\epsilon>0$, it is easy to check that $g\in \cal{G}'$.\\
Viceversa, if $g\in\cal{G}'$, it is enough to define $$v^{\epsilon}(x_1,s)=u^{\epsilon}(x_1,s)e_1-\int_0^{x_1}{(\partial_s u^{\epsilon}(\xi,s)\tau(s)+z^{\epsilon}(\xi,s)n(s))d\xi}$$
for every $(x_1,s)\in\omega$ and for every $\epsilon>0$. The conclusion follows then by Lemma \ref{scriverelemma}.
\end{oss}
\begin{oss}
The class $\cal{G}$ is always nonempty as it contains all functions $g\in L^2(\omega)$ which are affine with respect to $s$. Indeed, assume there exist $a_0,a_1\in L^2(0,L)$ such that $$g(x_1,s)=a_0(x_1)+s a_1(x_1)$$
for a.e. $(x_1,s)\in \omega$ and let $\capp{a}_i\in W^{1,2}(0,L)$ satisfying $\capp{a}'_i=a_i$, $i=0,1$. Then there exists $(\capp{a}^{\epsilon}_i)\subset C^{\infty}(0,L)$ such that $\capp{a}^{\epsilon}_i\ten \capp{a}_i$ strongly in $W^{1,2}(0,L)$ as $\epsilon\ten 0$, $i=0,1$ and setting
$$u^{\epsilon}(x_1,s):=\capp{a}^{\epsilon}_0(x_1)+s\capp{a}^{\epsilon}_1(x_1)$$
for every $(x_1,s)\in \omega$ and $z^{\epsilon}=0$ for every $\epsilon>0$, the claim follows by Remark \ref{recuse}.

 We also remark that if $g\in L^2(\omega)$ and there exist $\alpha_i\in L^2(0,L)$, $i=1,2,3$, such that 
 \begin{equation}
 \label{psg}
 \partial_s g=\alpha_1 N+\alpha_2\tau_2+\alpha_3\tau_3,
 \end{equation} then $g\in\cal{G}$. Indeed, by \eqref{psg} there exists $\alpha_4\in L^2(0,L)$ such that
 $$g=\alpha_1\int_0^s{N(\xi)d\xi}+\alpha_2\gamma_2+\alpha_3\gamma_3+\alpha_4.$$Let $\capp{\alpha}_i\in W^{1,2}(0,L)$ be such that $\capp{\alpha}_i'=\alpha_i$ for $i=1,2,3$. Then, setting
 $$u:=\capp{\alpha}_1\int_0^s{N(\xi)d\xi}+\capp{\alpha}_2\gamma_2+\capp{\alpha}_3\gamma_3+\capp{\alpha}_4,$$
 we have that $u\in W^{1,2}(\omega)$, $\partial_s^i u\in L^2(\omega)$ for $i=1,\cdots,6,$ and the map $z\in W^{1,2}(\omega)$, with $\partial_s^i z\in L^2(\omega)$ for $i=3,\cdots,5$, defined as $$z:=-\capp{\alpha}_1T-\capp{\alpha}_2\tau_3+\capp{\alpha}_3\tau_2,$$
 satisfies $\partial_s^2 u=kz$. For every $i=1,\cdots, 4$ there exists a sequence $(\alpha^{\epsilon}_i)\in C^{\infty}(0,L)$ such that $\alpha^{\epsilon}_i\ten \capp{\alpha}_i$ strongly in $W^{1,2}(0,L)$, as $\epsilon\ten 0$. Hence, defining
 $$u^{\epsilon}:={\alpha}^{\epsilon}_1\int_0^s{N(\xi)d\xi}+{\alpha}^{\epsilon}_2\gamma_2+{\alpha^{\epsilon}_3}\gamma_3+{\alpha^{\epsilon}_4},$$ 
 $$z^{\epsilon}:=-{\alpha^{\epsilon}_1}T-{\alpha^{\epsilon}_2}\tau_3+{\alpha^{\epsilon}_3}\tau_2,$$
 we have $\partial_1 u^{\epsilon}\ten g$ strongly in $L^2(\omega)$, $\partial_s^2 u^{\epsilon}=kz^{\epsilon}$ for every $\epsilon>0$, and both sequences $(u^{\epsilon})$ and $(z^{\epsilon})$ have the required regularity.
\end{oss}

\begin{oss}
The structure of the class $\cal{G}$ depends on the behaviour of the curvature $k$ of the curve $\gamma$.
 
For instance, if $k$ vanishes only at a finite number of points, then $\cal{G}=L^2(\omega)$. Indeed, let $0=p_0<p_1<\cdots<p_m=1$ be such that $k(s)\neq 0$ for every $s\in (p_i,p_{i+1})$, $i=0,\cdots,m-1$. For any $g\in L^2(\omega)$ there exists a sequence $(g^{\epsilon})\subset C^{\infty}_0\big((0,L)\times\bigcup_{i=0}^{m-1}(p_i,p_{i+1})\big)$ such that $g^{\epsilon}\ten g$ strongly in $L^2(\omega)$. Choosing $$u^{\epsilon}(x_1,s)=\int_0^{x_1}{g^{\epsilon}(\xi,s)d\xi}$$ for every $s\in (0,1)$, then $(u^{\epsilon})\subset C^{\infty}((0,L)\times(0,1))$ and for every $\epsilon>0$ there exists $\lambda^{\epsilon}>0$ such that 
$$2\lambda^{\epsilon}<\min_{i=0,\cdots,m-1}(p_{i+1}-p_i)$$ 
and $\partial_s^2 u^{\epsilon}=0$ in 
$$\bigcup_{i=0,\cdots,m-1}\big((p_i,p_i+\lambda^{\epsilon})\cup (p_{i+1}-\lambda^{\epsilon},p_{i+1})\big).$$
Setting 
$$z^{\epsilon}=\begin{cases}
\frac{\partial_s^2 u^{\epsilon}}{k}&\text{ in }(0,L)\times \bigcup_{i=0}^{m-1}{(p_i+\lambda^{\epsilon}, p_{i+1}-\lambda^{\epsilon})},\\
0&\text{ otherwise}
\end{cases}$$
we deduce immediately by Remark \ref{recuse} that $g\in\cal{G}$.

Assume instead that the sign of $k$ has the following behaviour: there exists a finite number of points $0=p_0<p_1<\cdots<p_m=1$ such that, for every $i=0,\cdots,m-1$, $k(s)>0$ for every $s\in (p_i,p_{i+1})$, or $k(s)<0$ for every $s\in (p_i,p_{i+1})$, or $k(s)=0$ for every $s\in (p_i,p_{i+1})$. In other words,
$$\{s\in [0,1]:k(s)=0\}=\bigcup_{i\in I_1}{[p_{i-1},p_i]}\cup\bigcup_{i\in I_2}\{p_i\}.$$ with $I_1\subset \{1,\cdots,m\}$, $I_2\subset \{0,\cdots,m\}$ disjoint. Then
\begin{equation}
\label{gafftratt}
\cal{G}:=\Big\{g\in L^2(\omega): g\text{ is affine in the }s\text{ variable in } (0,L)\times \bigcup_{i\in I_2}(p_i,p_{i+1})\Big\}.
\end{equation}
In particular, if $k\equiv 0$ on $[0,1]$, it follows that $\cal{G}$ is the set of all functions $g\in L^2(\omega)$ that are affine in the $s$ variable.

To prove \eqref{gafftratt}, assume for simplicity that $m=2$ and $\{s\in [0,1]: k(s)=0\}=[p_1,p_2]$. Let $g$ be affine in the $s$ variable in $(0,L)\times(p_1,p_2)$. Then, there exist $a,b\in L^2(0,L)$ such that
$$g(x_1,s)=a(x_1)+sb(x_1)$$
for a.e. $(x_1,s)\in (0,L)\times(p_1,p_2)$. Let now $0<\delta<\frac{L}{3}$ and let $\epsilon>0$. Arguing as in the proof of Lemma \ref{scriverelemma}, we extend $g$ to the set
$$\omega^{\delta}:=(-\delta, L+\delta)\times(-\delta,1+\delta)$$
and we define 
$$g^{\epsilon}(x_1,s)=\begin{cases}
a(x_1)+sb(x_1)&\text{ in }(0,L)\times(p_1-\epsilon, p_2+\epsilon),\\
g(x_1,s)&\text{ otherwise in }\omega^{\delta}.
\end{cases}$$  
It is easy to see that $g^{\epsilon}\ten g$ strongly in $L^2(\omega^{\delta})$ and $\partial_s^2 g^{\epsilon}=0$ in the sense of distributions in the set
$(0,L)\times(p_1-\epsilon,p_2+\epsilon)$ for every $\epsilon>0$.

 Fix $\epsilon$, let $0<\lambda<\tfrac{\epsilon}{2}$ and let $\rho\in C^{\infty}_0((-\lambda,\lambda)^2)$. We set $\capp{g}^{\epsilon}:=g^{\epsilon}\ast \rho$. Then $\capp{g}^{\epsilon}\in C^{\infty}(\lin{\omega})$ and $\partial_s^2\capp{g}^{\epsilon}=0$ in $(0,L)\times (p_1-\lambda,p_2+\lambda)$. 
Defining
$$u^{\epsilon}(x_1,s)=\int_0^{x_1}{\capp{g}^{\epsilon}(\xi,s)d\xi},$$ then $u^{\epsilon}\in C^{\infty}(\lin{\omega})$ and $\partial_s^2 u^{\epsilon}=0$ in $(0,L)\times (p_1-\lambda,p_2+\lambda)$. Hence, setting 
$$z^{\epsilon}=\begin{cases}
0&\text{ in }(0,L)\times (p_1-\lambda,p_2+\lambda)\\
\frac{\partial_s^2 u^{\epsilon}}{k}&\text{ otherwise, }
\end{cases}$$
the claim follows by Remark \ref{recuse}, considering a sequence of convolution kernels and applying a diagonal argument.

An easy adaptation of the previous argument leads to the proof of \eqref{gafftratt} in the general case.
\end{oss}

From here to the end of the section we shall assume that
\begin{equation}
\label{ls}
\exists \lim_{h\ten 0}\frac{\delta_h}{h^2}:=\lambda\quad\text{ and }\quad \exists\lim_{h\ten 0}\frac{\delta_h}{h^3}:=\mu.
\end{equation}
For any $0<\mu<+\infty$, we introduce the class 
\begin{multline}
\cal{C}_{\mu}:=\Big\{(g,b)\in L^2(\omega)\times L^2(\omega): \exists v\in L^2(\omega;\mathbb{R}^3)\,\text{ such that }\\\partial_s v\in L^2(\omega;\mathbb{R}^3), \,
\label{clbfin} 
\partial_s v\cdot {\tau}=0,\,  \partial_s(\partial_s v\cdot {n})=b\,\text{ and }
 \partial_1^2 v\cdot \tau+\mu\partial_s g=0\Big\},
\end{multline} 
where the last two equalities hold in the sense of distributions.

For $\mu=0$ we set
\begin{equation}
\label{c0}
\cal{C}_0:=\cal{G}\times \cal{B},
\end{equation}
where
\begin{multline}
\cal{B}:=\Big\{b\in L^2(\omega):\,\exists v\in L^2(\omega;\mathbb{R}^3)\text{ such that }\\\partial_s v\in L^2(\omega,\mathbb{R}^3),\, \label{clb0} 
\partial_s v\cdot {\tau}=0,\,\partial_s(\partial_s v\cdot {n})=b \text{ and } 
 \partial_1^2 v\cdot \tau=0\Big\},
\end{multline} 
and again the last two equalities hold in the sense of distributions.
\begin{oss}
\label{regol}
Let $b\in\cal{B}$ and let $v$ be as in \eqref{clb0}. Then the tangential component $v\cdot \tau$ belongs to $W^{3,2}(\st)$. Indeed, since $\partial_s(\partial_s v\cdot n)=b$ and $\partial_s v\in L^2(\st;\mathbb{R}^3)$, we deduce that $\partial_s^2(v\cdot n)\in L^2(\st)$. Since $\partial_s v\cdot \tau=0$, we have $\partial_s(v\cdot\tau)=k(v\cdot n)$ and then $\partial_s^2(v\cdot \tau), \partial_s^3(v\cdot \tau)\in L^2(\st)$. By the last condition in \eqref{clb0}, we have $\partial_1(v\cdot \tau)\in W^{-1,2}(\st)$, $\partial_1^2(v\cdot \tau) \in L^2(\st)$ and $\partial_s\partial_1(v\cdot \tau)=\partial_1\partial_s(v\cdot \tau)\in W^{-1,2}(\st)$. Therefore, by Lemma $\ref{lio}$, $\partial_1(v\cdot \tau)\in L^2(\st)$. Arguing analogously, by Lemma \ref{lio} we also have that $\partial_1\partial_s(v\cdot \tau)\in L^2(\st)$, therefore $v\cdot \tau\in W^{2,2}(\st)$ with $\partial_s^3 (v\cdot \tau)\in L^2(\st)$. Applying again Lemma \ref{lio}, it is straightforward to see that $v\cdot \tau\in W^{3,2}(\st)$. On the other hand, we have no regularity conditions on the derivatives with respect to $x_1$ of the normal component of $v$.

In the case where $\mu\neq 0$, if $(g,b)\in\cal{C}_{\mu}$ and $v$ is as in \eqref{clbfin}, then the regularity of $v\cdot \tau$ and $v\cdot n$ with respect to $s$ is the same as in the previous case. It is still true that $\partial_1(v\cdot \tau)\in L^2(\st)$ but, in general, one cannot guarantee that $v\cdot \tau\in W^{2,2}(\st)$.
\end{oss}
\begin{oss}
\label{rksyme}
A function $b\in L^2(\omega)$ belongs to $\cal{B}$ if and only if there exists a function $\phi\in L^2(\omega;\mathbb{R}^3)$, with $\phi\cdot \tau \in W^{3,2}(\omega)$, $\phi\cdot e_1\in W^{1,2}(\omega)$ and $\partial_s (\phi\cdot n), \partial_s^2 (\phi\cdot n)\in L^2(\omega)$, such that
\begin{equation}
\label{sb1}
e(\phi)=0
\end{equation}
and
\begin{equation}
\label{sb2}
\partial_s(\partial_s \phi\cdot n)=b.
\end{equation}
In other words, $\phi$ is an infinitesimal isometry of the cylindrical surface 
$$\Sigma:=\Big\{x_1e_1+\gamma(s):x_1\in (0,L), s\in (0,1)\Big\}$$
satisfying \eqref{sb2}.

We first observe that the regularity of $\phi$ is sufficient to guarantee that $e(\phi)$, defined as in \eqref{syg}, belongs to $L^2(\omega;\mathbb{M}^{2\times 2}_{\sym})$. Moreover, if $b\in L^2(\omega)$ and $v$ is as in \eqref{clb0}, then there exists $v_1\in W^{1,2}(\omega)$ such that
\begin{numcases}{}
\nonumber \partial_1 v_1=0,&\\
\nonumber \partial_s v_1=-\partial_1 v\cdot \tau.
\end{numcases}
The map $\phi:=v_1 e_1+v$ satisfies \eqref{sb1} and \eqref{sb2}. The converse statement is trivial.

Similarly, a pair $(g,b)\in L^2(\omega)\times L^2(\omega)$ belongs to $C_{\mu}$ if and only if there exists a function $\phi\in L^2(\omega;\mathbb{R}^3)$ with $\phi\cdot \tau \in W^{1,2}(\omega)$, $\partial_s^2(\phi\cdot \tau), \partial_s^3(\phi\cdot \tau)\in L^2(\omega)$, $\phi\cdot e_1 \in W^{1,2}(\omega)$ and $\partial_s(\phi\cdot n), \partial_s^2(\phi\cdot n)\in L^2(\omega)$, such that
$$e(\phi)=\Big(\begin{array}{cc}\mu g&0\\0&0\end{array}\Big)$$
and $$\partial_s(\partial_s \phi\cdot n)=b.$$
\end{oss}
\begin{oss}
As in the case of the class $\cal{G}$ introduced in \eqref{classg}, the structure of $\cal{B}$ and $\cal{C}_{\mu}$ depends on the behaviour of the curvature $k$ of $\gamma$. 

For instance, if $k\equiv 0$ on $[0,1]$, then $\cal{B}=L^2(\omega)$. Indeed, condition \eqref{sb1} implies in this case that there exist some $\alpha,\beta,\delta\in \mathbb{R}$ such that
$$\phi(x_1,s)=(\alpha s+\beta)e_1+(-\alpha x_1 +\delta)\tau+\phi_t(x_1,s)n$$
for a.e. $(x_1,s)\in\omega$, while condition \eqref{sb2} reads as $\partial_s^2 \phi_t=b$. Hence $\cal{B}=L^2(\omega)$. Similarly, it can be deduced that $C_{\mu}=\{g\in L^2(\omega): g\text{ is affine in }s\}\times L^2(\omega)$.

If, instead, $k(s)\neq 0$ for every $s\in[0,1]$, then $\cal{B}=\{b\in L^2(\omega): b\text{ is affine in }x_1 \}$. 
  \end{oss}
We conclude the section by proving some approximation results. The first result concerns the class $\cal{C}_{\mu}$ in the case $\mu\neq 0$.
 \begin{lem}
 \label{aplem}
 Let $(g,b)\in C_{\mu}$ with $\mu\neq 0$. Then, there exists a sequence $(\phi^{\epsilon})\subset C^5(\lin{\st};\mathbb{R}^3)$ such that
 \begin{equation}
 \label{1co}
 e(\phi^{\epsilon})=\Big(\begin{array}{cc}\partial_1\phi^{\epsilon}_1&0\\0&0\end{array}\Big)\ten \Big(\begin{array}{cc}\mu g&0\\0&0\end{array}\Big)
 \end{equation}
 strongly in $L^2(\omega;\mathbb{M}^{2\times 2}_{\sym})$ as $\epsilon\ten 0$ and
 \begin{equation}
 \label{2co}
 \partial_s(\partial_s \phi^{\epsilon}\cdot n)\ten b
 \end{equation}
 strongly in $L^2(\omega)$ as $\epsilon\ten 0$.
  \end{lem}
  \begin{oss}
  By Lemma \ref{aplem} it follows, in particular, that if $(g,b)\in C_{\mu}$ with $\mu\neq 0$, then $g\in \cal{G}$.
  \end{oss}
 \begin{proof}[Proof of Lemma \ref{aplem}.]
  Without loss of generality we may assume that $\mu=1$. By the definition of $C_{\mu}$ and by Remark \ref{rksyme} there exists $\phi\in L^2(\omega;\mathbb{R}^3)$ with $\phi\cdot \tau\in W^{1,2}(\omega)$, \mbox{$\partial_s^2 (\phi\cdot \tau)$}, $\partial_s^3 (\phi\cdot \tau)\in L^2(\omega)$, $\phi\cdot e_1\in W^{1,2}(\omega)$ and $\partial_s (\phi\cdot n), \partial_s^2 (\phi\cdot n)\in L^2(\omega)$, such that
  \begin{equation}
  \label{3co}
  e(\phi)=\Big(\begin{array}{cc}g&0\\0&0\end{array}\Big)
  \end{equation}
  and $\partial_s(\partial_s \phi\cdot n)=b$. By \eqref{3co} it follows that
  \begin{equation}
  \label{4co}
  \partial_1\phi\cdot \tau+\partial_s\phi_1=0.
  \end{equation}
  Hence, there exists $u\in W^{1,2}(\omega)$, with $\partial_1 u\in W^{1,2}(\omega)$ such that $\partial_1 u=\phi_1$ and 
  \begin{equation}
  \label{5co}
  \phi\cdot \tau+\partial_s u=0
  \end{equation}
  holds in the sense of $L^2(\omega)$. Indeed, by \eqref{4co}, if $\lin{u}\in W^{1,2}(\omega)$ satisfies $\partial_1 \lin{u}=\phi_1$, there exists $\varphi\in W^{1,2}(0,1)$ such that
  $$\phi\cdot \tau+\partial_s \lin{u}=\dot{\varphi}.$$
  Defining $u:=\lin{u}-\varphi$, then $u$ has the required properties.
  
  We set $v=(\phi\cdot \tau)\tau+(\phi\cdot n)n$. For the sake of simplicity, we divide the proof into two steps.\\
  {\bf Step 1.}\\
 We claim that we can always reduce to the case where \mbox{$u\in W^{4,2}(\st)$}, $v_s:=v\cdot\tau \in W^{3,2}(\st)$, and $v_t:=v\cdot n\in W^{2,2}(\st)$, with $\partial_1^i u,\partial_1^i v_t,\partial_1^i v_s, \partial_1^i g \in L^2(\st)$ for every $i\in\mathbb{N}$.
 
  Let $0<\delta<\frac{L}{3}$. Arguing as in the proof of Lemma \ref{scriverelemma} we may extend $u$ and $v$ to the set
 $$\omega_\delta:=(-\delta,L+\delta)\times(0,1)$$
 in such a way that, denoting by $\til{v}$ and $\til{u}$ the extended map and setting 
 $\til{g}=\partial_1^2 \til{u}$ and $\til{b}=\partial_s (\partial_s \til{v}\cdot n)$ in $\stee$, then $\til{g}$ and $\til{b}$ are respectively extensions of $g$ and $b$ to $\omega_{\delta}$. Moreover, $\til{u}\in W^{1,2}(\stee)$ with $\partial_1 \til{u}\in W^{1,2}(\stee)$, $\til{v}_s\in W^{1,2}(\stee)$ with $\partial_s^2 \til{v}_s, \partial_s^3 \til{v}_s\in L^2(\stee)$ and $\til{v}_t, \partial_s \til{v}_t, \partial_s^2 \til{v}_t \in L^2(\stee)$. Finally, by \eqref{3co} and \eqref{5co}, $(\til{u},\til{v})$ solves
 \begin{equation}
 \nonumber 
 \partial_s \til{v}\cdot \tau=0 \quad \text{ and }\quad \til{v}\cdot \tau+\partial_s \til{u}=0\quad \text{ in }\stee.
 \end{equation}

 We now mollify the functions $\til{u},\til{v}, \til{g},$ and $\til{b}$ with respect to the $x_1$ variable. Let $0<\epsilon<\delta$, let $(\rho^{\epsilon})\subset C^{\infty}_0(-\epsilon,\epsilon)$ be a sequence of convolution kernels and let
 \begin{numcases}{}
 \nonumber\til{u}^{\epsilon}(x_1,s):=(\til{u}(\cdot,s)\ast \rho^{\epsilon})(x_1),&\\
 \nonumber\til{v}^{\epsilon}(x_1,s):=(\til{v}(\cdot,s)\ast\rho^{\epsilon})(x_1),&\\
 \nonumber\til{b}^{\epsilon}(x_1,s):=(\til{b}(\cdot,s)\ast \rho^{\epsilon})(x_1),&\\ 
 \nonumber\til{g}^{\epsilon}(x_1,s):=(\til{g}(\cdot,s)\ast\rho^{\epsilon})(x_1)&
 \end{numcases} 
 for a.e. $(x_1,s)\in \st$ and for any $\epsilon$. Then $(\til{u}^{\epsilon},\til{v}^{\epsilon})$ solves
 \begin{equation}
 \nonumber
 \partial_1^2\til{u}^{\epsilon}=\til{g}^{\epsilon},\quad \partial_s \til{v}^{\epsilon}\cdot \tau=0,\quad
  \til{v}^{\epsilon}\cdot \tau+\partial_s \til{u}^{\epsilon}=0 \quad \text{ and }\quad
 \partial_s(\partial_s \til{v}^{\epsilon}\cdot n)=\til{b}^{\epsilon} 
\end{equation}
in $\st$ for every $\epsilon$. Moreover $\til{b}^{\epsilon}\ten \til{b}$ in $L^2(\st)$ and $\til{g}^{\epsilon}\ten \til{g}$ in $L^2(\st)$. Now, $(\til{v}^{\epsilon}_s)\subset W^{3,2}(\st)$ and $(\til{v}^{\epsilon}_t)\subset W^{2,2}(\st)$ with \mbox{$(\partial_1^i \til{v}^{\epsilon}_s),(\partial_1^i\til{v}^{\epsilon}_t)\subset L^2(\st)$} for every $i\in\mathbb{N}$. Therefore, \mbox{$(\partial_s \til{u}^{\epsilon})\subset W^{3,2}(\st)$}. Since $(\partial_1^i\til{u}^{\epsilon})\subset L^2(\st)$ for every $i\in\mathbb{N}$, then \mbox{$(\til{u}^{\epsilon})\subset W^{4,2}(\st)$} and the proof of the claim is completed.\\
{\bf Step 2.}\\
Assume now that \mbox{$u\in W^{4,2}(\st)$}, $v_s:=v\cdot\tau \in W^{3,2}(\st)$ and $v_t:=v\cdot n\in W^{2,2}(\st)$, with $\partial_1^i u,\partial_1^i v_t,\partial_1^i v_s, \partial_1^i g \in L^2(\st)$ for every $i\in\mathbb{N}$. Since \mbox{$v_t\in W^{2,2}(\st)$} there exists a sequence $(v_t^{\epsilon})\subset C^{\infty}(\lin{\st})$ such that 
\begin{equation}
\label{cvtn}
v_t^{\epsilon}\ten v_t
\end{equation} strongly in $W^{2,2}(\st)$.\\ Let $v^{\epsilon}_s\in C^5(\lin{\omega})$ be the solution of 
\begin{equation}
\label{6co}
\partial_s v^{\epsilon}_s=kv^{\epsilon}_t
\end{equation}
in $\omega$, with $\myintt{v^{\epsilon}_s(x_1,s)}\in C^{\infty}([0,L])$ for any $\epsilon>0$ and
\begin{equation}
\label{7co}
\myintt{v^{\epsilon}_s(x_1,s)}\ten \myintt{v_s(x_1,s)}
\end{equation} strongly in $W^{3,2}(0,L)$. By Poincar\'e inequality, we deduce
$$\|v^{\epsilon}_s-v_s\|_{L^2(\omega)}\leq C\Big(\Big\|\myintt{(v^{\epsilon}_s-v_s)}\Big\|_{L^2(\omega)}+\|k(v^{\epsilon}_t-v_t)\|_{L^2(\omega)}\Big)$$ 
and hence, by \eqref{cvtn} and \eqref{7co}
\begin{equation}
\label{8co}
 v^{\epsilon}_s \ten v_s \quad \text{ and }\quad \partial_s v^{\epsilon}_s \ten \partial_s v_s 
 \end{equation}
strongly in $L^2(\omega)$. Let ${u}^{\epsilon}\in C^6(\lin{\st})$ be the solution of 
\begin{equation}
\label{9co}
\partial_s u^{\epsilon}+v^{\epsilon}_s=0
\end{equation}
in $\omega$, with $\myintt{u^{\epsilon}(x_1,s)}\in C^{\infty}([0,L])$, 
\begin{equation}
\label{9cobis}
\myintt{{u}^{\epsilon}(x_1,s)}\ten \myintt{u(x_1,s)}
\end{equation}
strongly in $W^{4,2}(0,L)$. By Poincar\'e inequality,
\begin{multline*}
\|\partial_1^2\partial_s (u^{\epsilon}-u)\|_{L^2(\omega)}=\|\partial_1^2(v^{\epsilon}_s-v_s)\|_{L^2(\omega)}\\\leq C\Big(\Big\|\myintt{\partial_1^2(v^{\epsilon}_s-v_s)}\Big\|_{L^2(\omega)}+\|k\partial_1^2(v^{\epsilon}_t-v_t)\|_{L^2(\omega)}\Big)
\end{multline*}
which converge to zero due to \eqref{cvtn} and \eqref{7co}. Hence, by \eqref{9cobis} and by Poincar\'e inequality \begin{equation}
\label{10co}
\partial_1^2 u^{\epsilon}\ten \partial_1^2 u=g
\end{equation}
strongly in $L^2(\omega)$. Defining
$$\phi^{\epsilon}:=\partial_1 u^{\epsilon} e_1+v^{\epsilon},$$
then \eqref{1co} follows by \eqref{6co}, \eqref{9co} and \eqref{10co}. Moreover
\begin{equation}
\nonumber  \partial_s(\partial_s \phi^{\epsilon}\cdot n)=\partial_s^2 v^{\epsilon}_t+\dot{k}v^{\epsilon}_s+k\partial_s v^{\epsilon}_s.
\end{equation}
Therefore, \eqref{2co} follows from \eqref{cvtn} and \eqref{8co}, and the proof of the lemma is completed.
 \end{proof}
 The next lemma provides an approximation result for the elements of the class $\cal{B}$ introduced in \eqref{clb0}. We require an additional condition on the sign of the curvature. 
  \begin{lem}
 \label{aplem2}
 Assume there exists a finite number of points $0=p_0<p_1<\cdots<p_m=1$ such that, for every $i=0,\cdots,m-1$, $k(s)>0$ for every $s\in (p_i,p_{i+1})$, or $k(s)<0$ for every $s\in (p_i,p_{i+1})$ or $k(s)=0$ for every $s\in (p_i,p_{i+1})$.
 Let $b\in \cal{B}$. Then, there exists a sequence \mbox{$(\phi^{\epsilon})\subset C^5(\lin{\st};\mathbb{R}^3)$} such that
 \begin{equation}{}
 \label{Eqaf1} 
 e(\phi^{\epsilon})=0\quad \text{ for every }\epsilon>0
 \end{equation} 
 and
\begin{equation}
\label{Eqaf2}
\partial_s(\partial_s \phi^{\epsilon}\cdot n)\ten b 
\end{equation}
strongly in $L^2(\st)$ as $\epsilon\ten 0$.
\end{lem}
 \begin{proof}
 By definition of $\cal{B}$ there exists $v\in L^2(\st;\mathbb{R}^3)$, with $\partial_s v\in L^2(\st;\mathbb{R}^3)$, such that
 \begin{eqnarray}{}
  \label{Eqaf3}&&\partial_s v\cdot \tau=0, \\
 \label{Eqaf4}&&\partial_s(\partial_s v\cdot n)=b, \\
 \label{Eqaf5}&&\partial_1^2 v \cdot \tau=0. 
 \end{eqnarray}
 
Arguing as in Step 1 of the proof of Lemma \ref{aplem}, we may extend both $v$ and $b$ to the set $\omega_{\delta}:=(-\delta,L+\delta)\times(0,1)$ for $0<\delta<\frac{L}{3}$ and, up to a regularization in the $x_1$ variable, we may assume that $v_t:=v\cdot n\in W^{2,2}(\omega)$, $v_s:=v\cdot\tau\in W^{3,2}(\omega)$ and $\partial_1^i v_t,\partial_1^i v_s, \partial_1^i b\in L^2(\omega)$ for every $i\in\mathbb{N}$. Moreover, by \eqref{Eqaf5} there exist $\alpha_0,\alpha_1\in W^{3,2}(0,1)$ such that 
\begin{equation}
\label{Eqafvs}
 v_s(x_1,s)=\alpha_0(s)+x_1\alpha_1(s),
 \end{equation}
for a.e. $(x_1,s)\in\omega$.

Let $Z:=\{s\in [0,1]:k(s)=0\}$. By assumption, $Z$ is the union of a finite number of intervals with a finite number of isolated points.
 For simplicity, we divide the proof into three steps. We first consider the case where $Z$ is a finite union of points. In the second step, we assume $Z$ to be a finite union of closed intervals and in the third step we study the general case.\\
{ \bf Step 1.}\\
 Assume that $Z=\bigcup_{i\in I}\{p_i\}$ for some $I\subset \{0,\cdots,m\}$. By \eqref{Eqaf3} and \eqref{Eqaf5}, we have $$k\partial_1^2 v_t=0$$
 a.e. in $\omega$, which in turn gives 
 \begin{equation}
 \label{vtaff}
 \partial_1^2 v_t=0
 \end{equation} 
 a.e. in $\omega$. Hence, by \eqref{Eqaf3}, \eqref{Eqafvs}, and $\eqref{vtaff}$, there exist $\beta_0,\beta_1\in W^{2,2}(0,1)$ such that
 \begin{equation}
 \label{vtaffbis}
 v_t(x_1,s)=\beta_0(s)+x_1\beta_1(s)\quad \text{ and }\quad\dot{\alpha}_i(s)=k(s)\beta_i(s),\quad i=0,1,
 \end{equation}
 a.e. in $\omega$. Therefore there exist two sequences $(\beta^{\epsilon}_0),(\beta^{\epsilon}_1)\subset C^{\infty}([0,1])$ such that 
 \begin{equation}
 \label{Eqaf6}
 \beta^{\epsilon}_i\ten\beta_i
 \end{equation}
  strongly in $W^{2,2}(0,1)$, as $\epsilon\ten 0$, for $i=0,1$. Let $\alpha_i^{\epsilon}\in C^5([0,1])$ be the solution of
 \begin{equation}
 \label{Eqaf7}
 \dot{\alpha}_i^{\epsilon}=k\beta_i^{\epsilon}\quad \text{ in }(0,1)
 \end{equation}
 such that $\myintt{\alpha^{\epsilon}_i}=\myintt{\alpha_i}$ for every $\epsilon$, for $i=0,1$. By Poincar\'e inequality, we deduce
 $$\|\alpha^{\epsilon}_i-\alpha_i\|_{L^2(0,1)}\leq C\|k(\beta^{\epsilon}_i-\beta_i)\|_{L^2(0,1)},$$
 hence \eqref{Eqaf6} and \eqref{Eqaf7} imply
 \begin{equation}
 \label{Eqaf8}
 \alpha^{\epsilon}_i\ten\alpha_i
 \end{equation}
  strongly in $W^{1,2}(0,1)$, $i=0,1$. Taking $\phi^{\epsilon}_1\in C^6([0,1])$ to be a solution of 
  \begin{equation}
  \label{Eqaf8bis}
  \dot{\phi}^{\epsilon}_1=-\alpha^{\epsilon}_1
  \end{equation} for every $\epsilon$ and setting
 $$\phi^{\epsilon}:=\phi^{\epsilon}_1e_1+(\alpha_0^{\epsilon}+x_1\alpha_1^{\epsilon})\tau+(\beta_0^{\epsilon}+x_1\beta_1^{\epsilon})n,$$
we have that $\phi^{\epsilon}\in C^5(\lin{\omega},\mathbb{R}^3)$, \eqref{Eqaf1} holds owing to \eqref{Eqaf7} and \eqref{Eqaf8bis}, while convergence \eqref{Eqaf2} is a straightforward consequence of \eqref{Eqaf4}, \eqref{Eqafvs}, \eqref{vtaffbis}, \eqref{Eqaf6} and \eqref{Eqaf8}.
 \\{\bf Step 2.}\\
 Assume that $Z=[p_1,1]$, with $0< p_1<1$. By \eqref{Eqaf3} and \eqref{Eqaf5}, $\partial_1^2 v_t=0$ in $(0,L)\times(0,p_1)$. Arguing as in the proof of Lemma \ref{scriverelemma}, we define $\omega^{\delta}:=(-\delta,L+\delta)\times(-\delta,1+\delta)$ and we extend $v_t$ to the set $\omega^{\delta}$ for a suitable $\delta>0$ in such a way that $v_t\in W^{2,2}(\omega^{\delta})$ and $\partial_1^2 v_t=0$ in $(-\delta,L+\delta)\times(-\delta,p_1)$.
 
 We slightly modify the map $v_t$ close to the point $p_1$ so that it remains affine with respect to $x_1$ in a neighbourhood of this point. More precisely, for $\epsilon<\frac{\delta}{2}$, we set
 $$v_t^{\epsilon}(x_1,s):=v_t(x_1,s-\epsilon)\quad\text{ in }\omega^{\frac{\delta}{2}}.$$
 It is easy to see that $(v_t^{\epsilon})\subset W^{2,2}(\omega^{\frac{\delta}{2}})$, moreover
 $$v_t^{\epsilon}\ten v_t,\quad \partial_s v_t^{\epsilon}\ten \partial_s v_t\quad\text{ and }\partial_s^2 v_t^{\epsilon}\ten \partial_s^2 v_t $$
 strongly in $L^2(\omega^{\frac{\delta}{2}})$ as $\epsilon\ten 0$ and $$\partial_1^2 v_t^{\epsilon}=0 \quad\text{ in }(-\delta, L+\delta)\times(-\epsilon, p_1+\epsilon).$$
 
 To conclude, we regularize the sequence $(v^{\epsilon}_t)$ by mollification.
 Let $0<\lambda<\epsilon$ and let $\rho\in C^{\infty}_0((-\lambda,\lambda)^2)$. Defining $\til{v}^{\epsilon}_t:=v^{\epsilon}_t \ast \rho$, we have that $\til{v}^{\epsilon}_t\in C^{\infty}(\lin{\omega})$ and 
 \begin{equation}
 \label{daffvt}
 \partial_1^2 \til{v}^{\epsilon}_t=0 \quad \text{ in }(0,L)\times (0,p_1).
 \end{equation}
  Considering a sequence of convolution kernels and applying a diagonal argument we may also assume that
 \begin{equation}
 \label{Eqaf10} 
  \til{v}^{\epsilon}_t\ten v_t, \quad \partial_s \til{v}_t^{\epsilon}\ten \partial_s v_t\quad \text{ and }\quad\partial_s^2 \til{v}^{\epsilon}_t\ten \partial_s^2 v_t 
 \end{equation}
 strongly in $L^2(\omega)$ as $\epsilon\ten 0$. 
 
 By \eqref{daffvt}, for every $\epsilon$ we may choose a map $v_s^{\epsilon}\in C^5(\lin{\omega})$ such that $$\partial_s v_s^{\epsilon}=k\til{v}_t^{\epsilon},\quad \partial_1^2 v_s^{\epsilon}=0\quad \text{ and }\quad \myintt{v_s^{\epsilon}}=\myintt{v_s}\quad\text{ in }\omega.$$ 
 The conclusion of the lemma follows now arguing as in Step 1.
 
 The same argument applies to the case where $Z=[0,p_1]$, with $0<p_1<1$, choosing
 $$v_t^{\epsilon}(x_1,s):=v_t(x_1,s+\epsilon)\quad\text{ in }\omega^{\frac{\delta}{2}}$$
 and arguing as in the previous case.
 
 Finally, assume that $Z=[p_1,p_2]\cup[p_3,1]$ with $0<p_1<p_2<p_3<1$. Let $\varphi\in C_0^{\infty}(\mathbb{R})$ with $0\leq \varphi(s)\leq 1$ for every $s\in\mathbb{R}$, $\varphi(s)=1$ for all $s\in [p_2-\eta, p_2+\eta]$ and $\varphi(s)=0$ for $s \leq p_1+\eta$ or $s\geq p_3-\eta$ for some $\eta>0$ such that
 $\eta<\min\{p_1,p_2-p_1,p_3-p_2,1-p_3\}$. The argument shown at the beginning of this step applies now choosing
 $$v_t^{\epsilon}(x_1,s):=(1-\varphi(s))v_t(x_1,s-\epsilon)+\varphi(s)v_t(x_1,s+\epsilon)\quad \text{ in }\omega^{\frac{\delta}{2}}$$ for $\epsilon$ small enough.
 
 The case where $Z$ is a finite union of disjoint intervals is a simple adaptation of the previous cases. 
  \\{\bf Step 3.}\\
 Consider now the general case and assume there exist $I_1\subset\{1,\cdots,m\}$, $I_2\subset\{0,\cdots,m\}$ disjoint such that
 $$Z=\bigcup_{i\in I_1}{[p_{i-1},p_i]}\cup\bigcup_{i\in I_2}\{p_i\}.$$
 Then $\partial_1^2 v_t=0$ a.e. in $(0,L)\setminus\Big(\bigcup_{i\in I_1}{[p_{i-1},p_i]}\Big)$ and the thesis follows arguing as in Step 2. 
 \end{proof}
\section{Compactness results}
\label{compactness}
\noindent In this section we deduce some compactness properties for sequences of deformations $(y^h)$ satisfying the uniform energy estimate \eqref{conditionenergy}.

Assumption (H4) on $W$ provides us
with a control on the $L^2$ distance of the rescaled gradients from $SO(3)$.
Applying Theorem \ref{fjm} on a scale of order $\delta_h$, we can construct a sequence of approximating rotations $(R^h)$, whose
$L^2$ distance from the rescaled gradients is still of order $\epsilon_h$. 
Because of the different scaling of the cross-section diameter and the cross-section thickness the approximating rotations turn out to depend both on the mid-fiber coordinate $x_1$ and on the arc-length coordinate s. Moreover, the derivatives of $(R^h)$ in the two variables have a different order of decay, as $h\ten 0$. 

More precisely, we have the following result.
\begin{teo}
\label{apr}
 Assume that $\frac{\epsilon_h}{\delta_h}\ten 0$. Let $(y^h)$ be a sequence
of deformations in $ W^{1,2}(\Omega;\mathbb{R}^3)$ satisfying
\eqref{conditionenergy}.
Then, there exists a sequence of constant rotations $(P^h)$ and a sequence $(R^h)\subset C^{\infty}(\lin{\omega}; \mathbb{M}^{3 \times
3})$ with the following properties: setting \mbox{$Y^h:=(P^h)^Ty^h-c^h,$} where $(c^h)$ is any sequence of constants in $\mathbb{R}^3$, for every $h>0$ we have
\begin{eqnarray}
\label{apr6}&&
\|\nabla_{h,\delta_h}Y^h R_0^T-R^h\|_{L^2(\Omega)}\leq C\epsilon_h,\\
\label{skewavg}&&
\myintom{\Big(\nabla_{h,\delta_h}Y^h R_0^T-(\nabla_{h,\delta_h}Y^h R_0^T)^T\Big)}=0,\\
\label{apr1}
  &&R^h(x_1,s)\in SO(3) \text{ for every }(x_1,s)\in \lin{\omega},\\
\label{apr3}
&&\|R^h-Id\|_{L^2(\omega)}\leq C\frac{\epsilon_h}{\delta_h},\\
\label{apr4}
&&\|\partial_1 R^h\|_{L^2(\omega)}\leq C \frac{\epsilon_h}{\delta_h},\\
\label{apr5}
&&\|\partial_{s} R^h\|_{L^2(\omega)}\leq C \frac{h\epsilon_h}{\delta_h}.
\end{eqnarray}
\end{teo}
\begin{proof}
By \eqref{conditionenergy} and (H4), the sequence $(y^h\comp(\psi^h)^{-1})$ satisfies
\begin{equation}
\label{limen}
\int_{\Omega_h}{\dist^2(\nabla(y^h\comp(\psi^h)^{-1}), SO(3))dx}\leq Ch\delta_h\epsilon_h^2.
\end{equation}

Let us consider the sets
\begin{multline}
\nonumber
A_h^i:=\Big\{x_1 e_1+h\gamma(s)+\delta_h t
n(s): x_1\in\big(\tfrac{i_1L}{\eta_h},\tfrac{(i_1+1)L}{
\eta_h}\big),\,\\s\in
\big(\tfrac{i_2}{k_h},\tfrac{(i_2+1)}{k_h}\big),\, t\in\big(-\tfrac{1}{2},\tfrac{1}{2}\big)\Big\},
\end{multline}
where $\eta_h=\big[\frac{L}{\delta_h}\big]$, $k_h=\big[\frac{h}{\delta_h}\big]$ and $i=(i_1,i_2)$, with
\mbox{$i_1=0,\cdots,\eta_h-1$},
 \mbox{$i_2=0,\cdots,k_h-1$}. By Theorem \ref{fjm} and Remark \ref{ofjm} there exist a sequence of constant rotations $(\lin{Q}^i_h)\subset SO(3)$ and a 
constant $C$ independent of $h$ and $i$ satisfying
\begin{equation}
\label{ex}
\int_{A_h^i}{|\nabla (y^h\comp(\psi^h)^{-1})-\lin{Q}^i_h|^2 dx}\leq
C\int_{A_h^i}{\dist^2(\nabla (y^h\comp(\psi^h)^{-1}),SO(3)) dx}.
\end{equation}
To see that $C$ does not depend on $h$, we first notice that each set $A^i_h$
has the same rigidity constant of the set $\til{A}^i_h$ that is obtained by a
uniform dilation of $A^i_h$ of factor $\frac{1}{\delta_h}$.
Defining $\phi^i_h:(0,1)^3\longrightarrow \til{A}^i_h$ as
$$\phi^i_h(x_1,s,t)=\big(
\tfrac{(i_1+x_1)L}{\eta_h
\delta_h},\,\tfrac{h}{\delta_h}\gamma\big(\tfrac{i_2+s}{k_h}\big)
+\big(t-\tfrac{1}{2}\big)n\big(\tfrac{i_2+s}{k_h}\big)\big),$$ we
conclude that the sets $\til{A}^i_h$ are the image of the unitary cube through a family of
uniformly bi-Lipschitz transformations. Therefore by Remark \ref{ofjm} the constant $C$ is the same
for every $i$ and for every $h$.

 Let \mbox{$Q^h:\omega\longrightarrow SO(3)$} be the
piecewise 
constant map given by $Q^h(x_1,s):=\lin{Q}_h^i$ for $(x_1,s)\in \big(\frac{i_1L}{\eta_h},\frac{(i_1+1)L}{
\eta_h}\big)\times \big(\frac{i_2}{k_h},\frac{i_2+1}{k_h}\big)$ where $i_1=0,\cdots,\eta_h-1$ and $i_2=0,\cdots,k_h-1$.
Summing \eqref{ex} over $i$, changing variables and using \eqref{uncon}, we deduce that
\begin{equation}
\label{1stima}
\int_{\Omega}{|\nh y^h R_0^T-Q^h|^2 dx}\leq C\int_{\Omega}{\dist^2(\nh y^h R_0^T, SO(3))dx.}\leq C\epsilon_h^2.
\end{equation}
 
 Consider the set
 \begin{multline}
 \nonumber
 B_h^i:=\Big\{x_1 e_1+h\gamma(s)+\delta_h t
n(s): x_1\in\big((i_1-1)\tfrac{L}{\eta_h},(i_1+2)\tfrac{L}{
\eta_h}\big),\,\\s\in
\big((i_2-1)\tfrac{1}{k_h},(i_2+2)\tfrac{1}{k_h}\big),\, t\in\big(-\tfrac{1}{2},\tfrac{1}{2}\big)\Big\},
\end{multline}
for $i_1=1,\cdots,\eta_h-2$, $i_2=1,\cdots,k_h-2$, and for every $h>0$.
Applying the rigidity estimate to the sets $B^i_h$ we obtain that for every $(i_1,i_2)$ there exists a map $\capp{Q}_h^i\subset SO(3)$ satisfying
$$\int_{B_h^i}{|\nabla (y^h\comp(\psi^h)^{-1})-\capp{Q}^i_h|^2 dx}\leq
C\int_{B_h^i}{\dist^2(\nabla (y^h\comp(\psi^h)^{-1}),SO(3)) dx}.$$
Let now $j_k$ be an integer in the set $\{i_k-1,i_k,i_k+1\}$, $k=1,2$ and let $j=(j_1,j_2)$. Since $A^j_h\subset B^i_h$, there holds
\begin{multline}
\cal{L}^3(A^j_h)\big|Q^h\big(\tfrac{j_1L}{\eta_h},
\tfrac{j_2}{k_h}\big)-
\capp{Q}_h^i\big|^2
\leq
2\int_{A_h^j}{\big|Q^h\big(\tfrac{j_1L}{\eta_h},\tfrac{j_2}{
k_h}\big)-\nabla (y^h\comp(\psi^h)^{-1})\big|^2 dx}\\ 
\label{darichiamare} + 2\int_{B_h^i}{|\nabla (y^h\comp(\psi^h)^{-1})-\capp{Q}_h^i|^2 dx}\leq C\int_{B_h^i}{\dist^2(\nabla
(y^h\comp(\psi^h)^{-1}), SO(3)) dx}.
 \end{multline}
Then, by \eqref{limen} we have
\begin{equation}
\label{ptwq}
\cal{L}^3(A^i_h)\big|Q^h\big(\tfrac{(i_1\pm 1)L}{\eta_h},
\tfrac{i_2\pm 1}{k_h}\big)-
{Q}_h\big(\tfrac{i_1L}{\eta_h},
\tfrac{i_2}{k_h}\big)\big|^2
\leq Ch\delta_h\epsilon_h^2,
\end{equation}
for any $i_1=1,\cdots \eta_h-1, i_2=1,\cdots k_h-1$.

We first extend the map $Q^h$ to the strip $\mathbb{R}\times (0,1)$ by setting
\begin{equation}
\nonumber
Q^h(x_1,s)=\begin{cases}
Q^h(0,s) &\text{ if }(x_1,s)\in (-\infty,0)\times (0,1),\\
Q^h(L,s) &\text{ if }(x_1,s)\in (L,+\infty)\times (0,1),\\
\end{cases}
\end{equation}
and then to the whole $\mathbb{R}^2$ by
\begin{equation}
\nonumber
Q^h(x_1,s)=\begin{cases}
Q^h(x_1,0) &\text{ if }(x_1,s)\in \mathbb{R}\times(-\infty,0),\\
Q^h(x_1,1) &\text{ if }(x_1,s)\in \mathbb{R}\times(1,+\infty).
\end{cases}
\end{equation}
Since $Q^h$ is constant in each set $A^i_h$, inequality \eqref{ptwq} yields\begin{equation}
 \label{pointvariationQ}
|Q^h(x_1+\xi,s+\lambda)-Q^h(x_1,s)|^2\leq C
\frac{h\epsilon_h^2}{\delta_h^2}
\end{equation}
for every \mbox{$(x_1,s)\in \omega$} and for \mbox{$|\xi|\leq \frac{L}{\eta_h}$},
\mbox{$|\lambda|\leq\frac{1}{k_h}$}. 
Moreover, since $Q^h$ is piecewise constant, \eqref{darichiamare} and \eqref{ptwq} imply\begin{multline}
\label{2darich}
\int_{\big(\frac{i_1L}{\eta_h},\frac{(i_1+1)L}{\eta_h}\big)\times\big(\frac{i_2}{k_h},\frac{i_2+1}{
k_h}\big)}
{|Q^h(x_1+\xi,s+\lambda)-Q^h(x_1,s)|^2 dx_1 ds}\\
\leq
\frac{C}{h\delta_h}\int_{B^i_h}{\dist^2\big(\nabla(y^h\comp(\psi^h)^{-1}),SO(3)
\big) } ,
\end{multline}
for every $i_1=1,\cdots ,\eta_h-2$, $i_2=1,\cdots ,k_h-2$. 

Let now $\omega'\subset\subset \omega$. For $h$ small enough, there holds
$\omega'\subset \big(\frac{L}{\eta_h},L-\frac{L}{\eta_h}\big)\times\big(\frac{1}{\eta_h},1-\frac{1}{\eta_h}\big).$
Hence, by \eqref{limen} and \eqref{2darich}, since $x\in \Omega_h$ belongs to at most $9$ sets of
the form $B_h^i$, summing over the $i_k's$, we deduce
\begin{equation}
\label{integralestimaterot}
\int_{\omega'}
{|Q^h(x_1+\xi,s+\lambda)-Q^h(x_1,s)|^2 dx_1ds}\leq C \epsilon_h^2,
\end{equation}
for all $|\xi|\leq \delta_h$, $|\lambda|\leq \frac{\delta_h}{h}$.

To obtain a $C^{\infty}$ sequence of rotations, we regularize $(Q^h)$ by means of convolution kernels.
Let $\eta\in C^{\infty}_{0}(0,1)$, $\eta\geq 0$, $\int_{0}^{1}{\eta(s)ds}=1$.
We define
$\varphi^h(\xi,\lambda):=\frac{h}{\delta_h^2}\eta\big(\frac{\xi}{\delta_h}
\big)\eta\big(\frac{h\lambda}{\delta_h}\big)$ for every $\xi\in(0,\delta_h),\lambda\in\big(0,\frac{\delta_h}{h}\big)$ and we notice that, for $h$ small enough, supp $\varphi^h$ is contained into a ball whose radius is smaller than the distance between $\omega'$ and the boundary of $\omega$.

Setting $\til{Q}^h:=Q^h\ast \varphi^h$, by Holder inequality and \eqref{integralestimaterot} we have  
\begin{equation}
\nonumber
\int_{\omega'}{|\til{Q}^h(x_1,s)-Q^h(x_1,
s)|^2 dx_1ds}\leq C\epsilon_h^2,
\end{equation}
which implies that
\begin{equation}
\label{rm}
\|\til{Q}^h-Q^h\|_{L^2(\omega)}\leq C\epsilon_h
\end{equation}
since the constant $C$ does not depend on the choice of $\omega'$.
Analogously we
obtain
\begin{equation}
\label{b1}
\|\partial_1 \til{Q}^h\|_{L^2(\omega)}\leq C\frac{\epsilon_h}{\delta_h}
\end{equation}
and
\begin{equation}
\label{b2}
\|\partial_s \til{Q}^h\|_{L^2(\omega)}\leq C \frac{h\epsilon_h}{\delta_h}.
\end{equation}

Finally, let $U$ be a neighbourhood of $SO(3)$ where the projection \mbox{$\Pi:U\longrightarrow SO(3)$} 
is well defined and regular. By \eqref{pointvariationQ}, we deduce
\begin{eqnarray}
\label{auxex} &&|\til{Q}^h(x_1,s)-Q^h(x_1,s)|^2\leq \|\varphi^h\|^2_{L^2\big((0,\delta_h)\times(0,\tfrac{\delta_h}{h})\big)}
\frac{\delta_h^2}{h}\frac{h\epsilon_h^2}{\delta_h^2}\leq
C\frac{h\epsilon_h^2}{\delta_h^2},
\end{eqnarray}
for every $(x_1,s)\in \omega$.
Since $\frac{\epsilon_h}{\delta_h}\ten 0$,  
$\til{Q}^h \in U$ for $h$ small enough and, thus, we can define $\til{R}^h:=\Pi(\til{Q}^h)$. It is immediate to see that, for every $h>0$, $\til{R}^h$ satisfies \eqref{apr1}. Furthermore, by \eqref{b1} and \eqref{b2} and by regularity of $\Pi$, \eqref{apr4} and \eqref{apr5} hold.
By definition of $\til{R}^h$,
\begin{equation}
\label{2stima}
\|\til{R}^h-\til{Q}^h\|_{L^2(\omega)}\leq\|Q^h-\til{Q}
^h\|_{L^2(\omega)}
\end{equation}
therefore \eqref{apr6} follows from \eqref{1stima} and \eqref{rm}.

By Poincar\'e inequality, 
given
$$\lin{R^h}:=\fint_{\omega}{\til{R}
^hdx_1ds},$$ \eqref{apr4} and \eqref{apr5} yield
$$\|\til{R}^h-\lin{R}^h\|_{L^2(\omega)}\leq C\|\nabla \til{R}^h\|_{L^2(\omega)}\leq C\frac{\epsilon_h}{\delta_h}.$$
This implies that $\dist(\lin{R}^h, SO(3))\leq C\frac{\epsilon_h}{\delta_h}$. Hence,
there 
exists a sequence of constant rotations $(S^h)\in SO(3)$ such that
$|\lin{R}^h-S^h|\leq 
C\frac{\epsilon_h}{\delta_h}$, which in turn implies
\begin{equation}
\label{num}
\|\til{R}^h-S^h\|_{L^2(\omega)}\leq C\frac{\epsilon_h}{\delta_h}.
\end{equation}
We define $\hat{R}^h:=(S^h)^T\til{R}^h$ and $\hat{y}^h=(S^h)^T y^h$. By the properties of the sequence $(\til{R}^h)$ and by \eqref{num}, $\hat{R}^h$ satisfies \eqref{apr6} and \eqref{apr1}--\eqref{apr5}. 

To provide a sequence of rotations satisfying also \eqref{skewavg}, we argue as in \cite[Lemma 3.1]{F-M-P} and we introduce the matrices
$$F^h:=\fint_{\Omega}{\nabla_{h,\delta_h}\hat{y}^h R_0^Tdx_1dsdt}.$$
We notice that
\begin{equation}
\label{firstest}|F^h-Id|\leq \fint_{\Omega}{|\nabla_{h,\delta_h}\hat{y}^h R_0^T-Id|dx_1dsdt}\leq C\frac{\epsilon_h}{\delta_h},
\end{equation}
as $\hat{R}^h$ satisfies \eqref{apr6} and \eqref{apr3}.
It turns out that $\det F^h>0$ for $h$ small enough, therefore by polar decomposition Theorem, for every $h$ there exist $P^h\in SO(3)$ and $U^h\in \mthree_{sym}$ such that $$F^h=P^hU^h,$$ and \begin{equation}\label{interest}|U^h-Id|=\dist(F^h,SO(3))\leq |F^h-Id|.\end{equation} The symmetry of $U^h$, together with \eqref{firstest} and \eqref{interest}, yields for any $h>0$\begin{equation}
\label{lastest}
|P^h-Id|\leq |P^h-U^h|+|U^h-Id|\leq C|F^h-Id|\leq C\frac{\epsilon_h}{\delta_h}.
\end{equation} 
Defining $R^h:=(P^h)^T\hat{R}^h$ and $Y^h:=(P^h)^T\hat{y}^h$, then \eqref{apr6}, \eqref{apr1}, \eqref{apr4} and \eqref{apr5} follow immediately. Moreover, since
$$\|R^h-Id\|_{L^2(\omega)}\leq \|R^h-\hat{R}^h\|_{L^2(\omega)}+\|\hat{R}^h-Id\|_{L^2(\omega)}\leq C(\|P^h-Id\|_{L^2(\omega)}+\|\hat{R}^h-Id\|_{L^2(\omega)}),$$
then \eqref{apr3} holds due to \eqref{lastest} and from the fact that $\hat{R}^h$ satisfies \eqref{apr3}. Finally, by symmetry of $U^h$, for every $h>0$
$$\myintom{(\nabla_{h,\delta_h}Y^h R_0^T-(\nabla_{h,\delta_h} Y^h R_0^T)^T)}$$$$=\cal{L}^3(\Omega)((P^h)^TF^h-(F^h)^T P^h)=\cal{L}^3(\Omega)(U^h-(U^h)^T)=0,$$ which concludes the proof of \eqref{skewavg} and of the proposition.
\end{proof}

From now on we shall refer to the sequence of deformations $(Y^h)$ introduced in Theorem \ref{apr}, where the constants $c^h$ are chosen in such a way to satisfy 
\begin{equation}
 \label{aver}
\myintom{(Y^h-\psi^h)}=0.
\end{equation}
We introduce the tangential derivative of the tangential displacement, associated with $Y^h$, given by
\begin{equation}
\label{defu}
 g^h(x_1,s,t):=\frac{1}{\epsilon_h}\partial_1(Y^h_1-\psi^h_1),\end{equation}
for a.e. $(x_1,s,t)\in \Omega$, 
and the (averaged) twist function, associated with $Y^h$, given by
\begin{equation}
\label{defw}
 w^h(x_1,s):=\frac{\delta_h}{h\epsilon_h}\myintr{\partial_s (Y^h-\psi^h)\cdot
n},
\end{equation}
for a.e. $(x_1,s)\in\omega$. 

We are now in a position to prove the first compactness result.
\begin{teo}
 \label{disp}
Under the same assumptions of Theorem \ref{apr}, let $(R^h)$ and $(Y^h)$ be the sequences introduced in Theorem \ref{apr}, with $(c^h)$ such that \eqref{aver} holds.
Then
\begin{equation}
\label{ss}
Y^h\ten x_1e_1 \text{ strongly in }W^{1,2}(\Omega;\mathbb{R}^3).
\end{equation}
Let $(g^h)$ and $(w^h)$ be the sequences defined in \eqref{defu} and \eqref{defw}. Then
there exist $g\in L^2(\Omega)$ and $w\in W^{1,2}(0,L)$ such that, up to subsequences, 
\begin{eqnarray}
 \label{u1}
&&g^h\deb g \quad\text{ weakly in }L^2(\Omega) \text{ if }\eqref{l}\text{ holds},\\
\label{w}
&&w^h\ten w \quad\text{ strongly in }L^2(\omega),\\
 \label{convah}
&& A^h:=\frac{\delta_h}{\epsilon_h}(R^h-Id)\deb A \quad\text{ weakly in }W^{1,2}(\omega;\mathbb{M}^{3\times 3}),\\
\label{convgradhdh}&& 
 \frac{\delta_h}{\epsilon_h}(\nabla_{h,\delta_h} Y^h
R_0^T-Id)\ten A \quad\text{ strongly in }L^2(\Omega;\mathbb{M}^{3\times 3}),\\
\label{symmAh}&& \frac{\delta_h^2}{\epsilon_h^2}\,\sym{(R^h-Id)}\ten \frac{A^2}{2} \quad\text{ strongly in }L^2(\omega;\mathbb{M}^{3\times 3}),
\end{eqnarray}
where
\begin{equation}
\label{charA}
 A(x_1)=w(x_1)(e_3\otimes e_2-e_2\otimes e_3)
 \end{equation}
for a.e. $x_1\in (0,L)$. Moreover, $Y^h$ satisfies 
\begin{equation}
\label{symgr}
\|\sym (\nh Y^hR_0^T-Id)\|_{L^2}\leq C\Big(\epsilon_h+\frac{\epsilon_h^2}{\delta_h^2}\Big).
\end{equation}
Finally, there exists $b\in L^2(\omega)$ such that, setting 
\begin{equation}
\label{charB}
 B(x_1,s)=\Bigg(\begin{array}{ccc}
 0 & w'(x_1)\tau_3(s)&-w'(x_1)\tau_2(s)\\
 -w'(x_1)\tau_3(s) & 0 & -b(x_1,s)\\
  w'(x_1)\tau_2(s)&  b(x_1,s)& 0
 \end{array}\Bigg)
\end{equation}
for a.e. $(x_1,s)\in\omega$, we have, up to subsequences,
\begin{equation}
 \label{eqb}\frac{\delta_h}{h\epsilon_h}\partial_{s}R^h\deb B \quad\text{
weakly in }L^2(\omega;\mathbb{M}^{3\times 3}).
\end{equation}
\end{teo}
\begin{proof}
By properties \eqref{apr3}, \eqref{apr4} and \eqref{apr5}, the sequence $(A^h)$
is uniformly bounded in
$W^{1,2}(\omega;\mathbb{M}^{3\times 3})$. Therefore, there
exists $A\in W^{1,2}(\omega;\mathbb{M}^{3 \times 3})$ such that, up to subsequences, \eqref{convah} holds.
 Since $\|\partial_{s}A^h\|_{L^2}\leq Ch$ by \eqref{apr5}, we have that $A=A(x_1)$.
 
By Sobolev embedding theorems,
convergence of $A^h$ is actually strong in $L^q(\omega;\mthree)$ for every $q\in [1,+\infty)$ and since
\begin{equation}
 \sym\, A^h=-\frac{\epsilon_h}{\delta_h}\frac{(A^h)^T A^h}{2},
\end{equation}
\eqref{symmAh} follows immediately. 

By \eqref{apr6} and by strong convergence
of $(A^h)$ in $L^2$, we obtain \eqref{convgradhdh}. In particular, \mbox{$\partial_1 Y^h\ten e_1$} and \mbox{$\partial_s Y^h,\partial_t Y^h \ten 0$} strongly in $L^2(\Omega;\mathbb{R}^3)$. \eqref{ss} follows now owing to \eqref{aver} and Poincar\'e inequality. Moreover,
$$\|\sym(\nh Y^h R_0^T-Id)\|_{L^2}\leq\|\sym(\nh Y^hR_0^T-R^h)\|_{L^2}+\|\sym(R^h-Id)\|_{L^2}.$$
Hence, \eqref{symgr} holds due to \eqref{apr6} and \eqref{symmAh}.

By \eqref{apr5},
there exists a map $B\in L^2(\omega;\mathbb{M}^{3\times 3})$ satisfying \eqref{eqb}.
Differentiating the identity $$(R^h)^TR^h=Id,$$ we obtain
$$(\partial_{s} R^h)^T (R^h-Id)+(R^h-Id)^T
\partial_{s}R^h=-2\sym \,\partial_{s}R^h.$$
Then, by \eqref{convah} and \eqref{eqb}, we deduce that $B$ is
skew-symmetric.
 
 We claim that
 \begin{equation}
 \label{ab}
 Be_1=A'\tau.
 \end{equation} Indeed, let $\varphi \in W^{1,2}_{0}(\Omega;\mathbb{R}^3)$. Then
 \begin{multline}
\scal{\frac{\delta_h}{h\epsilon_h}\partial_{s}\partial_1(Y^h-\psi^h)}{
\varphi}_{W^{-1,2}\times W^{1,2}_0}=\\
\label{fcolb}-\myintom{\frac{\delta_h}{h\epsilon_h}(\nabla_{h,\delta_h}
Y^h-R^hR_0)e_1\cdot \partial_{s}\varphi }
 +
 \displaystyle{\myintom{\frac{\delta_h}{h\epsilon_h}\partial_{s}R^h
e_1\cdot\varphi }}.
 \end{multline}
 The first term in \eqref{fcolb} is infinitesimal due to \eqref{apr6}, while  \eqref{eqb} yields
$$\myintom{\frac{\delta_h}{h\epsilon_h}\partial_{s}R^h e_1\cdot \varphi }\ten
\myintom{Be_1\cdot\varphi}.$$
 On the other hand, we have 
\begin{multline}
\nonumber
\scal{\frac{\delta_h}{h\epsilon_h}\partial_{s}\partial_1(Y^h-\psi^h)}{
\varphi}_{W^{-1,2}\times W^{1,2}_0}=\\-\myintom{\frac{\delta_h(h-\delta_h t k)}{h\epsilon_h}(\nh Y^h-R_0)e_2\cdot\partial_1\varphi},
\end{multline}
 which in turn gives 
\begin{equation}
\label{dsd1h}
\scal{\frac{\delta_h}{h\epsilon_h}\partial_{s}\partial_1(Y^h-\psi^h)}{
\varphi}_{W^{-1,2}\times W^{1,2}_0}\ten\myintom{A'\tau\cdot\varphi}.
\end{equation}
 owing to \eqref{convgradhdh} and \eqref{uncon}. Combining \eqref{fcolb} and \eqref{dsd1h}, we obtain \eqref{ab}. 
 
 Since $B$ is skew-symmetric, the following equality holds
 $$0=B_{11}(x_1,s)=A'_{12}(x_1)\tau_2(s)+A'_{13}(x_1)\tau_3(s),$$ for a.e. $x_1 \in
(0,L)$ and $s \in (0,1)$. This last condition, together with the assumption that $k$ is not identically zero, implies \begin{equation}\label{constants}A_{12}'=A_{13}'\equiv 0.\end{equation}
On the other hand, by \eqref{skewavg} and \eqref{convgradhdh} we deduce that
$$\myintol{A(x_1)}=0.$$
Hence, $A_{12}=A_{13}=0$.

To conclude the proof of the Theorem, we consider the sequences $(g^h)$ and $(w^h)$. To prove \eqref{u1}, we notice that 
\begin{equation}
\label{deru}
g^h=\frac{1}{\epsilon_h}\Big((\partial_1 Y^h_1-R^h_{11})+(R^h_{11}-1)\Big).
\end{equation}
Since we are assuming that \eqref{l} holds, then by \eqref{apr6} and \eqref{symmAh}, $(g^h)$ is uniformly bounded in $L^2(\Omega)$.
Therefore, there exists $g\in L^2(\Omega)$ such that
\eqref{u1} holds up to subsequences.

 As for the twist function, by \eqref{uncon} and \eqref{convgradhdh}, 
$$\displaystyle{\frac{\delta_h}{h\epsilon_h}\partial_s(Y^h-\psi^h)\ten A\tau \text{ strongly in } L^2(\Omega;\mathbb{R}^3)}$$ therefore \eqref{w} follows. In particular, $w=A_{32}$, hence $w\in W^{1,2}(0,L)$ and \eqref{charA} holds. Finally, by \eqref{ab} we deduce the representation \eqref{charB}. 
\end{proof}
In the next proposition we show further compactness properties of the twist functions $w^h$, under stronger assumptions on the order of decay of $\epsilon_h$ with respect to the cross-sectional thickness $\delta_h$.
\begin{comment}
If in addition we are in the periodic case, then further constraints must be satisfied by the limiting functions $w$ and $b$.
\end{comment} 
\begin{prop}
\label{fpwb}
Under the same assumptions of Theorem \ref{disp}, let $w^h$ and $b$ be the functions introduced in \eqref{defw} and \eqref{charB}. 
If ${\frac{\epsilon_h}{h\delta_h}\ten 0}$, we have
\begin{equation}
\label{cwb}
\frac{1}{h}{\partial_{s}w^h}\deb b \text{ weakly in }W^{-1,2}(\omega).
\end{equation}
 \end{prop}
\begin{proof}
Assume that ${\frac{\epsilon_h}{h\delta_h}\ten 0}$. By definition of the functions $w^h$, we have 
\begin{equation}
\label{connwb}
\frac{1}{h}\partial_{s}w^h=\frac{\delta_h}{h^2\epsilon_h}\partial_s \myintr{(\nh Y^h-R^hR_0)e_2\cdot n(h-\delta_h t k)}+\frac{\delta_h}{h\epsilon_h}\partial_s((R^h-Id)\tau\cdot n).
\end{equation} 
By \eqref{uncon} and \eqref{apr6}, the first term on the right-hand side of \eqref{connwb} converges to zero strongly in \mbox{$W^{-1,2}(\omega)$}. The second term can be further decomposed as
\begin{eqnarray}
\nonumber && \frac{\delta_h}{h\epsilon_h}\partial_s((R^h-Id)\tau\cdot n)=\frac{\delta_h}{h\epsilon_h}\partial_s R^h \tau\cdot n+\frac{\delta_h}{h\epsilon_h}(R^h n\cdot n-R^h\tau\cdot\tau)k.
\end{eqnarray} Hence, \eqref{cwb} follows from \eqref{symmAh}, \eqref{charB} and \eqref{eqb}. \end{proof}
\section{Characterization of the limit strain and liminf inequality}
\label{characterization}
\noindent In this section we shall prove a liminf inequality for the rescaled energies $\frac{1}{\epsilon_h^2}\cal{J}^h$ defined in \eqref{defjh}.  To this purpose we introduce the strains:
\begin{equation} 
\label{Gh}
 G^h:=\frac{1}{\epsilon_h}\big((R^h)^T \nabla_{h,\delta_h}Y^h R_0^T-Id\big),
\end{equation}
where $(R^h)$ and $(Y^h)$ are the sequences introduced in Theorem \ref{apr},
and we prove their convergence to a limit strain $G$.
In Theorem \ref{deep} we deduce a characterization of $G$, together with some further properties of the limit functions $g$, $w$, and $b$ introduced in \eqref{u1}, \eqref{w}, and \eqref{charB}.

We begin with a characterization of $g$.
\begin{prop}
\label{ging}
Under the same assumptions of Theorem \ref{disp}, let \eqref{l} be satisfied. Let $g$ be the function introduced in \eqref{u1} and let $\cal{G}$ be the class defined in \eqref{classg}. Then $g\in\cal{G}$.
\end{prop}
\begin{proof}
Let $(Y^h)$ be as in Theorem \ref{disp}. For every $h>0$ let 
$$v^h:=\frac{1}{\epsilon_h}\myintr{(Y^h_1-x_1)e_1}+\frac{h}{\epsilon_h}\myintr{\big((Y_2^h-\psi_2^h)e_2+(Y_3^h-\psi_3^h)e_3\big)}.$$
By definition, $v^h\in W^{1,2}(\omega;\mathbb{R}^3)$ for every $h>0$; moreover by \eqref{symgr}, we have
$$\Big\|\frac{h}{\epsilon_h}\partial_s (Y^h-\psi^h)\cdot \tau\Big\|_{L^2(\Omega)}=\Big\|\frac{h(h-\delta_h t k)}{\epsilon_h}(\nh Y^h R_0^T-Id)\tau\cdot\tau\Big\|_{L^2(\Omega)}\leq Ch^2,$$
which implies
\begin{equation}
\nonumber
\partial_s {v}^h\cdot \tau \ten 0 \quad\text{ strongly in }L^2(\omega).
\end{equation} 
 Similarly, by \eqref{symgr} we deduce 
\begin{equation}
\nonumber
\partial_s v_1^h+\partial_1 {v}^h\cdot \tau \ten 0\quad \text{ strongly in }L^2(\omega).
\end{equation}  
By \eqref{defu} and \eqref{u1} we also have
$$\partial_1 v^h_1\deb g \text{ weakly in }L^2(\omega).$$
The thesis follows now by Lemma \ref{scriverelemma}.
\end{proof}
  We are now in a position to state the first theorem of this section. For any $M\in\mthree$ we use the notation $M_{tan}$ to denote the matrix
$$M_{tan}:=(e_1|\tau)^T(Me_1|M\tau).$$
 \begin{teo}
\label{deep}
Let the assumptions of Theorem \ref{apr} be satisfied. Assume in addition \eqref{l}. Let $(Y^h)$ and $(R^h)$ be as in Theorem \ref{disp} and let $G^h$ be defined as in \eqref{Gh}. Then there
exists $G\in L^2(\Omega;\mathbb{M}^{3\times 3})$ such that, up to subsequences,
\begin{equation}
 \label{conGh}
G^h\deb G \text{ weakly in }L^2(\Omega;\mthree).
\end{equation}
Let $g,w,b$  be the maps introduced in \eqref{u1}, \eqref{w}, and \eqref{charB}.  Then
\begin{equation}
\label{tan}
 G_{tan}(x_1,s,t)=-t\Bigg (\begin{array}{cc}0&w'(x_1)\\w'(x_1)&b(x_1,s)\end{array}\Bigg)+G_{tan}(x_1,s,0)
\end{equation}
for a.e. $(x_1,s,t)\in \Omega$ and
\begin{equation}
\label{g11u}
(G_{tan})_{11}=G_{11}=g
\end{equation}
a.e. in $\Omega$.

If in addition \eqref{ls} holds, then: 
\begin{enumerate}[a)]
\item if $\mu=+\infty$, there exist $\alpha_1,\alpha_2,\alpha_3\in L^2(0,L)$ such that
\begin{equation}
\label{dsg}
\partial_s g=\alpha_1 N+\alpha_2 \tau_2+\alpha_3\tau_3;
\end{equation}
\item
if $\lambda=+\infty$, then \eqref{dsg} holds with $\alpha_1=0;$
\item
if $0<\lambda<+\infty$, then $w\in W^{2,2}(0,L)$ and \eqref{dsg} holds with $\alpha_1=\frac{1}{\lambda}w'';$
\item
if $\lambda=0$, then $w''=0;$
\item
if $0\leq\mu<+\infty$, then $(g,b)\in\cal{C}_{\mu}$,
where $\cal{C}_{\mu}$ is the class defined in \eqref{clbfin}--\eqref{c0}.
\end{enumerate}
  \end{teo}
\begin{proof}
By \eqref{apr6}, the sequence $(G^h)$ is uniformly bounded in $L^2(\Omega;\mthree)$; therefore there exists $G\in L^2(\Omega;\mthree)$ such that \eqref{conGh} holds. By \eqref{conGh}, $$\partial_t(R^hG^hR_0e_1)\deb \partial_t Ge_1$$
weakly in $W^{-1,2}(\Omega;\mathbb{R}^3)$. On the other hand, by \eqref{convgradhdh} we have
\begin{equation}
\nonumber 
\partial_t(R^hG^hR_0e_1)=\frac{1}{\epsilon_h}\,\partial_t(\nh Y^h-R^hR_0)e_1=\frac{1}{\epsilon_h}\,\partial_t(\partial_1 Y^h)=\frac{\delta_h}{\epsilon_h}\,\partial_1\Big(\frac{\partial_t Y^h}{\delta_h}\Big)\ten A'n
\end{equation}
strongly in $W^{-1,2}(\Omega)$. Hence,
\begin{equation}
\label{s1}
G(x_1,s,t)e_1=tA'(x_1)n(s)+G(x_1,s,0)e_1
\end{equation}
for a.e. $(x_1,s,t)\in \Omega$.

 To characterize $G\tau$ we observe that
\begin{multline}
\nonumber \partial_t (R^hG^hR_0e_2)=\frac{1}{\epsilon_h}\,\partial_t(\nh Y^h-R^hR_0)e_2\\=\frac{1}{\epsilon_h}\,\partial_t\big(\frac{\partial_s(Y^h-\psi^h)}{h-\delta_h t k}\big)
=\frac{1}{\epsilon_h}\,\frac{\delta_h}{h-\delta_h t k}\,\partial_s\big(\frac{\partial_t(Y^h-\psi^h)}{\delta_h}\big)+\frac{\delta_h k}{\epsilon_h(h-\delta_h t k)}\,\frac{\partial_s(Y^h-\psi^h)}{h-\delta_h t k}\\
=  \frac{1}{\epsilon_h}\,\frac{\delta_h}{h-\delta_h t k}\,\big(\partial_s(\nh Y^h-R^hR_0)e_3+k(\nh Y^h-R^hR_0)e_2\big)
+\frac{1}{\epsilon_h}\,\frac{\delta_h}{h-\delta_h t k}\,(\partial_s R^h)n.
\end{multline} 
The first term on the right hand side of the previous equality is converging to zero strongly in $W^{-1,2}(\Omega;\mathbb{R}^3)$ due to \eqref{apr6}, therefore by \eqref{uncon} and \eqref{eqb} we deduce 
$$\partial_t (R^hG^h R_0 e_2)\deb Bn$$
weakly in $W^{-1,2}(\Omega;\mathbb{R}^3)$. On the other hand, by \eqref{conGh} we have
$$\partial_t(R^hG^hR_0e_2)\deb \partial_tG \tau$$
weakly in $W^{-1,2}(\Omega;\mathbb{R}^3)$. Hence
\begin{equation}
\label{s2}
G(x_1,s,t)\tau(s)=tB(x_1,s)n(s)+G(x_1,s,0)\tau(s)
\end{equation}
for a.e. $(x_1,s,t)\in \Omega$. Combining \eqref{charA}, \eqref{charB}, \eqref{s1} and \eqref{s2}, we obtain \eqref{tan}. 

By \eqref{symmAh} and \eqref{Gh},
 \begin{equation}
 \label{g11uc}
 \frac{1}{\epsilon_h}\partial_1 (Y^h_1-\psi^h_1)\deb G_{11}=(G_{tan})_{11}
 \end{equation}
 weakly in $L^2(\Omega)$. Therefore \eqref{g11u} follows from \eqref{defu} and \eqref{u1}.
    
 To prove the properties a)--e), we first claim that
\begin{equation}
\label{motiv}
\frac{(h-\delta_h t k)}{\epsilon_h}\partial^2_1(Y^h-\psi^h)\cdot \tau \deb -\partial_s g \text{ weakly in }W^{-1,2}(\Omega).
\end{equation}
Indeed, by \eqref{symgr}
 \begin{multline}
 \nonumber  \Big\|\frac{h-\delta_h t k}{\epsilon_h}\Big(\frac{\partial_s (Y^h_1-\psi^h_1)}{h-\delta_h t k}+\partial_1 (Y^h-\psi^h) \cdot \tau\Big)\Big\|_{L^2}\\\leq 2\Big\|\frac{h-\delta_h t k}{\epsilon_h}\sym (\nh Y^hR_0^T-Id)\Big\|_{L^2}
\leq Ch\Big(1+\frac{\epsilon_h}{\delta_h^2}\Big),
 \end{multline}
 which converges to zero by \eqref{l}.
 Therefore, \eqref{motiv} follows by \eqref{g11uc} and \eqref{g11u}.
 
We introduce the maps 
$\lin{v}^h\in W^{1,2}(\Omega;\mathbb{R}^2)$, given by 
\begin{equation}
\label{vh}
\lin{v}^h:=\Big(\begin{array}{c}v^h_2\\v^h_3\end{array}\Big)=\frac{h}{\epsilon_h}\Big(\begin{array}{c}Y^h_2-\psi^h_2\\Y^h_3-\psi^h_3\end{array}\Big)
\end{equation} for every $h>0$.
       By \eqref{uncon} and \eqref{motiv}, we have 
       \begin{equation}
       \label{cd1vh}
       \partial_1^2 \lin{v}^h\cdot \tg\deb -\partial_s g \text{ weakly in }W^{-1,2}(\Omega).
       \end{equation}
       
Let $\eh$ be the operator introduced in \eqref{eeps}, with $\epsilon$ replaced by $\frac{\delta_h}{h}$. By straightforward computations and by \eqref{symgr}, we obtain 
\begin{equation}
\label{4.14bis}
\|\sym(\eh \lin{v}^h\lin{R}_0^T)\|_{L^2(\Omega;\mathbb{M}^{2\times 2})}\leq\frac{h^2}{\epsilon_h}\Big\|\sym(\nh Y^hR_0^T-Id)\Big\|_{L^2(\Omega;\mthree)}\leq Ch^2
\end{equation}
for every $h>0$.
Applying Korn's inequality \eqref{korn} and using the notation of Theorem \ref{Korn}, we deduce
\begin{equation}
\label{ehbd}
\|\lin{v}^h-\ph(\lin{v}^h)\|_{W^{1,2}(S;\mathbb{R}^2)}\leq C\frac{h}{\delta_h}\|\sym(\eh v^h \lin{R}_0^T)\|_{L^2(S;\mathbb{M}^{2\times 2})},
\end{equation}
  for a.e. $x_1\in (0,L)$. Integrating \eqref{ehbd} with respect to $x_1$, by \eqref{4.14bis} it follows that
     \begin{eqnarray}
   \label{fe} \|\lin{v}^h-\ph(\lin{v}^h)\|_{L^2(\Omega;\mathbb{R}^2)}\leq C\frac{h^3}{\delta_h},\\
   \label{se} \|\partial_s(\lin{v}^h-\ph(\lin{v}^h))\|_{L^2(\Omega;\mathbb{R}^2)}\leq C\frac{h^3}{\delta_h},\\
 \label{te} \|\partial_t(\lin{v}^h-\ph(\lin{v}^h))\|_{L^2(\Omega;\mathbb{R}^2)}\leq C\frac{h^3}{\delta_h}.
\end{eqnarray} 
By Lemma \ref{struttmel}, for every $h>0$ there exist $\alpha_1^h,\alpha_2^h,\alpha_3^h\in L^2(0,L)$ such that $\ph(\lin{v}^h)$ has the following structure:
 \begin{equation}
 \label{struttphw}
 \ph(\lin{v}^h)=\Big(\begin{array}{c}\alpha^h_2\\\alpha^h_3\end{array}\Big)+\alpha^h_1\Big(\begin{array}{c}-\gamma_3\\\gamma_2\end{array}\Big)-\frac{\delta_h}{h}t\alpha^h_1\tg.
 \end{equation}
 Moreover,
 \begin{equation}
\label{4.19bis}
\frac{\delta_h}{h^2}\myintr{\partial_s \lin{v}^h\cdot \n}=w^h
\end{equation}
 for every $h>0$ and for a.e. $(x_1,s)\in\omega$. On the other hand, by \eqref{struttphw}
 \begin{equation}
 \label{4.19ter}
 \frac{\delta_h}{h^2}\myintr{\partial_s \ph(\lin{v}^h)\cdot \n}=\frac{\delta_h}{h^2}\alpha_1^h
 \end{equation}
for every $h>0$ and for a.e. $(x_1,s)\in\omega$. Therefore, by estimate \eqref{se}, there holds
\begin{equation}
\label{fnum}
\|\alpha^h_1-\frac{h^2}{\delta_h}w^h\|_{L^2(\omega)}\leq C\frac{h^3}{\delta_h},
\end{equation}
which in turn gives
\begin{equation}
\label{easycons}
\frac{\delta_h}{h}t\alpha^h_1\tg\ten 0
\end{equation}
strongly in $L^2(\Omega)$.

We first consider the case where $\mu=+\infty$. Then, by \eqref{cd1vh} and \eqref{fe}, we have 
\begin{equation}
 \label{stpoint}
 \partial_1^2(\ph(\lin{v}^h))\cdot \tg\deb -\partial_s g\smallskip\text{ weakly in }W^{-2,2}(\Omega).
 \end{equation}
 Hence, by \eqref{struttphw}, \eqref{easycons} and by Lemma \ref{conv2e} there exist $\alpha_1,\alpha_2,\alpha_3\in L^2(0,L)$ such that \eqref{dsg} holds and the proof of a) is completed.
 
 The proof of b) follows immediately by \eqref{fnum} as if $\lambda=+\infty$, then $\alpha_1=0$.
 
Consider now the case where $\lambda<+\infty$. By \eqref{cd1vh} and \eqref{fe}, we deduce
 \begin{equation}
 \label{f1}
 \frac{\delta_h}{h^2}\partial_1^2(\ph(\lin{v}^h))\cdot \tg\deb -\lambda\partial_s g\smallskip\text{ weakly in }W^{-2,2}(\Omega)
 \end{equation}
 for every $\lambda<+\infty$.
 By \eqref{struttphw}, \eqref{easycons} and by Lemma \ref{conv2e}, there exist $\beta_1,\beta_2,\beta_3\in L^2(0,L)$ such that 
 \begin{equation}
 \label{eth2}
 \frac{\delta_h}{h^2}(\alpha_i^h)''\deb \beta_i,\,i=1,2,3\end{equation}
 weakly in $W^{-2,2}(0,L)$ and 
 \begin{equation}
 \label{4.19quater}
 \lambda \partial_s g=\beta_1 N+\beta_2\tau_2+\beta_3\tau_3.
 \end{equation} 
 By \eqref{fnum} and \eqref{eth2}, if $0<\lambda<+\infty$ we obtain $\beta_1=w''$ and $w\in W^{2,2}(0,L)$. This proves c). 
 
 Finally, to prove d) we observe that if $\lambda=0$, by \eqref{4.19quater} and by Lemma \ref{conv2e} we have $\beta_1=\beta_2=\beta_3=0$, hence $w''=0$.
 
  Assume now that $0<\mu<+\infty$. Defining $\capp{v}^h:=\lin{v}^h-\ph(\lin{v}^h)$, by \eqref{fe}--\eqref{te} there exists $\capp{v}\in L^2(\Omega;\mathbb{R}^2)$ with $\partial_s\capp{v},\partial_t\capp{v} \in L^2(\Omega;\mathbb{R}^2)$ such that, up to subsequences 
\begin{eqnarray}
\label{cv1}&&\capp{v}^h\deb \capp{v},\\
\label{cv2}&&\partial_s\capp{v}^h\deb\partial_s\capp{v},\\
\label{cv3}&&\partial_t \capp{v}^h\deb\partial_t\capp{v},
\end{eqnarray}
weakly in $L^2(\Omega;\mathbb{R}^2)$. Since 
\begin{equation}
\label{id}
\sym(\eh \capp{v}^h \lin{R}_0^T)=\sym(\eh v^h\lin{R}_0^T),
\end{equation} for every $h>0$, combining \eqref{4.14bis} with equations \eqref{cv1}--\eqref{cv3}, we deduce
\begin{equation}
\label{sistlw}
\partial_s\capp{v}\cdot \tg=0,\,\quad\text{ and }\quad
 \partial_t\capp{v}=0. 
\end{equation}
By \eqref{cwb} and \eqref{4.19bis}, we have 
 \begin{equation}
 \label{tbn}
 \frac{\delta_n}{h^3}\partial_s \myintr{\partial_s \lin{v}^h\cdot \n}\deb b
 \end{equation}
 weakly in $W^{-1,2}(\omega)$.
 On the other hand, by \eqref{4.19ter},
 \begin{equation}
 \label{tbn1}\frac{\delta_h}{h^3}\partial_s \myintr{\partial_s \lin{v}^h\cdot \n}=\frac{\delta_h}{h^3}\partial_s \myintr{(\partial_s \lin{v}^h-\partial_s \ph(\lin{v}^h))\cdot \n} =\frac{\delta_h}{h^3}\partial_s\Big(\myintr{\partial_s \capp{v}^h\cdot \n}\Big),
 \end{equation}
 therefore \eqref{sistlw} yields 
 \begin{equation}
 \label{lastconlw}
 \mu\partial_s(\partial_s \capp{v}\cdot\n)=b,
 \end{equation}
 whenever $0<\mu<+\infty$. By \eqref{cd1vh}, \eqref{fe} and \eqref{cv1},
$$\partial_1^2(\ph(\lin{v}^h))\cdot\tg\deb -\partial_s g-\partial_1^2 \capp{v}\cdot\tg$$
weakly in $W^{-2,2}(\Omega)$. By Lemma \ref{conv2e}, by \eqref{struttphw} and \eqref{easycons}
there exist $\alpha_1,\alpha_2,\alpha_3\in W^{-2,2}(0,L)$ such that
\begin{equation}
\label{fff}
\partial_s g=-\alpha_2\tau_2-\alpha_3\tau_3+\alpha_1N-\partial_1^2 \capp{v}\cdot \tg.
\end{equation}
For $i=1,2,3$, let now $\capp{\alpha}_i\in L^2(0,L)$ be such that $(\capp{\alpha}_i)''=\alpha_i$ and let
$$v=\mu\Bigg(\begin{array}{c}0\\\capp{v}+\Big(\begin{array}{c}\capp{\alpha}_2\\\capp{\alpha}_3\end{array}\Big)+\capp{\alpha}_1\Big(\begin{array}{c}-\gamma_3\\\gamma_2\end{array}\Big)\end{array}\Bigg).$$
By \eqref{sistlw}, \eqref{lastconlw}, and \eqref{fff} we deduce that $$\partial_s v\cdot {\tau}=0, \quad \partial_s(\partial_s v\cdot n)=b,\quad \text{ and }\partial_1^2 v\cdot \tau+\mu\partial_s g=0,$$ where the last two equalities hold in the sense of distributions. Therefore $(g,b)\in \cal{C}_{\mu}$.

Finally, we study the case where $\mu=0$. For every $h>0$, we define 
\begin{equation}
\label{tvh}
\til{v}^h:=\frac{\delta_h}{h^3}\capp{v}^h.
\end{equation}
By \eqref{4.14bis}, 
\begin{equation}
\label{ehbd2}
\|\sym(\eh\tvh\lin{R}_0^T)\|_{L^2}\leq C\frac{\delta_h}{h}.
\end{equation}
By \eqref{fe}--\eqref{te} there exists $\til{v}\in L^2(\Omega;\mathbb{R}^2)$, with $\partial_s \til{v},\partial_t \til{v} \in L^2(\Omega;\mathbb{R}^2)$, such that, up to subsequences, \begin{eqnarray}
&&\label{cv10}\capp{v}^h\deb \til{v} ,\\
&&\label{cv20}\partial_s\capp{v}^h\deb\partial_s\til{v},\\
&&\label{cv30}\partial_t \capp{v}^h\deb\partial_t\til{v},
\end{eqnarray} weakly in $L^2(\Omega;\mathbb{R}^2)$. By \eqref{tbn}, \eqref{tbn1} and \eqref{ehbd2},
$\til{v}$ satisfies 
\begin{equation}
\label{sistlb}\partial_s\til{v}\cdot\tg=0,\quad \partial_t\til{v}=0\quad \text{ and }\partial_s(\partial_s\til{v}\cdot\n)=b.
\end{equation}
Moreover, by \eqref{cd1vh} and by \eqref{cv10} we deduce that
\begin{equation}
\label{tbn2}
\partial_1^2(\ph(\tvh))\cdot\tg\deb-\partial_1^2\til{v}\cdot\tg
\end{equation}
weakly in $W^{-2,2}(\Omega)$. Hence, by \eqref{easycons}, Lemma \ref{struttmel} and Lemma \ref{conv2e}, there exist $\alpha_1,\alpha_2,\alpha_3\in W^{-2,2}(0,L)$ such that
\begin{equation}
\label{eqaff}
\partial_1^2\til{v}\cdot\tg=-\alpha_2\tau_2-\alpha_3\tau_3+\alpha_1N.
\end{equation}
Let now $\capp{\alpha}_1,\capp{\alpha}_2, \capp{\alpha}_3\in L^2(0,L)$ be such that 
$\alpha_1=(\capp{\alpha}_1)''$, $\alpha_2=(\capp{\alpha}_2)''$ and $\alpha_3=(\capp{\alpha}_3)''$.
Defining $$v=\Bigg(\begin{array}{c}0\\\til{v}+\Big(\begin{array}{c}\capp{\alpha}_2\\\capp{\alpha}_3\end{array}\Big)+\capp{\alpha}_1\Big(\begin{array}{c}-\gamma_3\\\gamma_2\end{array}\Big)\end{array}\Bigg),$$
by \eqref{sistlb} and \eqref{eqaff}, we deduce that $$\partial_s v\cdot \tau=0, \quad \partial_s(\partial_s v\cdot n)=b \quad \text{ and }\quad\partial_1^2 v\cdot \tau=0,$$ where the last two equalities hold in the sense of distributions. This concludes the proof of the theorem.\end{proof}
We can now deduce a lower bound for the rescaled energies $\epsilon_h^{-2}\cal{J}^h$. To this purpose, from here to the end of the paper we shall assume that \eqref{ls} holds and we introduce the classes $\cal{A}_{\lambda,\mu}$ defined as follows. We define
\begin{multline}
\label{ainin}
  \cal{A}_{\infty,\infty}:=\big\{(w,g,b)\in W^{1,2}(0,L)\times L^2(\st)\times L^2(\st):\\ \partial_s g=\alpha_2\tau_2+\alpha_3\tau_3, 
 \text{ with }\alpha_i\in L^2(0,L),\,i=2,3\big\}.
 \end{multline}
 For $\lambda\in (0,+\infty)$ we set
  \begin{multline}
  \label{alinf}
 \cal{A}_{\lambda,\infty}:=\big\{(w,g,b)\in W^{2,2}(0,L)\times L^2(\st)\times L^2(\st):\\ \partial_s g=\tfrac{1}{\lambda}w'' N+\alpha_2\tau_2+\alpha_3\tau_3, 
\text{ with }\alpha_i\in L^2(0,L),\,i=2,3\big\},
\end{multline}
and for $\lambda=0$
\begin{multline}
\label{a0in}
\cal{A}_{0,\infty}:=\big\{(w,g,b)\in W^{2,2}(0,L)\times L^2(\st)\times L^2(\st):
w''=0\text{ and }\\\partial_sg=\alpha_1 N+\alpha_2\tau_2+\alpha_3\tau_3,
 \text{ with }\alpha_i\in L^2(0,L),\,i=1,2,3\big\}.
 \end{multline}
 Finally, for $\mu\in[0,+\infty)$ let
 \begin{equation}
 \label{a0m}
 \cal{A}_{0,\mu}:=\big\{(w,g,b)\in W^{2,2}(0,L)\times \cal{C}_{\mu}:w''=0\big\}.
 \end{equation}
 We consider the functionals $\cal{J}_{\lambda,\mu}:W^{1,2}(0,L)\times L^2(\st)\times L^2(\st)\longrightarrow [0,+\infty]$, defined as
  \begin{equation}
 \label{jlm}
 \cal{J}_{\lambda,\mu}(w,g,b):=\frac{1}{24}\myintss{Q_2(s,w',b)}+\frac{1}{2}\myintss{\mathbb{E}g^2}
\end{equation}
for $(w,g,b)\in\cal{A}_{\lambda,\mu}$, and  $\cal{J}_{\lambda,\mu}(w,g,b)=+\infty$ otherwise.
\begin{teo}
\label{liminft}
Assume \eqref{l} and \eqref{ls}. Let $\cal{A}_{\lambda,\mu}$ be the classes defined in \eqref{ainin}--\eqref{a0m}. Given any sequence of deformations $(y^h)\subset W^{1,2}(\Omega;\mathbb{R}^3)$ satisfying \eqref{conditionenergy}, there exist rotations $P^h\in SO(3)$ and constants $c^h\in\mathbb{R}^3$ such that, setting $Y^h:=(P^h)^Ty^h-c^h$ and defining $g^h$ and $w^h$ as in \eqref{defu} and \eqref{defw}, there exist $(g,w,b)\in \cal{A}_{\lambda,\mu}$ such that, up to subsequences,
\begin{numcases}{}
\nonumber g^h\deb g\text{ weakly in }L^2(\Omega),&\\
\nonumber w^h\ten w\text{ in }L^2(\st),&\\
\frac{1}{h}\partial_s w^h\deb b\text{ weakly in }W^{-1,2}(\st).&\label{prop1}
\end{numcases}
Moreover,
\begin{equation}
\label{liminf}
 \liminf_{h\ten 0}\frac{1}{\epsilon_h^2}{\cal{J}^h(Y^h)}
\geq  \cal{J}_{\lambda,\mu}(w,g,b),
\end{equation}
where $ \cal{J}_{\lambda,\mu}$ is the functional defined in \eqref{jlm}.
\end{teo}
\begin{proof}
Convergence properties \eqref{prop1} follow from Theorem \ref{disp} and Proposition \ref{fpwb}. Moreover, Proposition \ref{ging} and Theorem \ref{deep} guarantee that $(g,w,b)\in \cal{A}_{\lambda,\mu}$.
The proof of the lower bound \eqref{liminf} is an adaptation of \cite[Proof of Corollary 2]{F-J-M}. 

Let $G^h$ be defined as in \eqref{Gh}. We introduce the functions $$\chi^h(x):=\begin{cases}
 1 &\text{ if } |G^h|< \frac{1}{\sqrt{\epsilon_h}}\\0 &\text{ otherwise. }\end{cases}$$
 It is easy to see that $\chi^h \ten 1$ in measure and $\chi^hG^h \deb G$ weakly in $L^2(\Omega;\mthree)$.
 By frame indifference of $W$,
 \begin{eqnarray}
 \nonumber
 \liminf_{h\ten 0}\frac{\cal{J}^h(Y^h)}{\epsilon_h^2}&=&\liminf_{h\ten 0}\frac{1}{\epsilon_h^2}\myintom{W(\nabla_{h,\delta_h}Y^h R_0^T)}\\
 \nonumber
 &=&\liminf_{h \ten 0}\frac{1}{\epsilon_h^2}\myintom{W(Id+\epsilon_h G^h)}\\
 \label{feq}&\geq& \liminf_{h \ten 0} \frac{1}{\epsilon_h^2}\myintom{\chi^h W(Id+\epsilon_h G^h)}.
 \end{eqnarray}
 
 Owing to assumptions (H2), (H3), and (H5), by a Taylor expansion of $W$ around the identity we have: $$W(Id+F)=\frac{1}{2}Q_3(F)+\eta(F),$$
 for any $F\in\mthree$, where $\frac{\eta(F)}{|F|^2}\ten 0$ as $|F|\ten 0$. Setting $\xi(t):=\sup_{|F|\leq t}\frac{\eta(F)}{|F|^2}$, then $\xi(t)\ten 0$ as $t\ten 0$ and
 $$\chi_h W(Id+\epsilon_hG^h)\geq\chi_h\frac{\epsilon_h^2}{2}Q_3(G^h)-\chi_h\epsilon_h^2\xi({\epsilon_h}|G^h|)|G^h|^2.$$
 
 Thus, we can continue the chain of inequalities in \eqref{feq} as
 \begin{eqnarray}
 \nonumber \liminf_{h\ten 0}\frac{\cal{J}^h(Y^h)}{\epsilon_h^2}&\geq&  \liminf_{h \ten 0}\Big\{\frac{1}{2}\myintom{Q_3(\chi_h G^h)}-\frac{1}{2}\myintom{\chi_h\xi(\epsilon^h |G^h|)|G^h|^2}\Big\}.\\\label{feq2}
  \end{eqnarray}
  By the assumptions on $W,$ $Q_3$ is a positive semi-definite quadratic form, hence the first term in \eqref{feq2} is lower semicontinuous with respect to weak convergence in $L^2$. By definition of the sequence $(\chi_h)$ and by uniform boundedness of $\|G^h\|_{L^2(\Omega;\mthree)}$, the second term in \eqref{feq2} can be bounded as
  $$\frac{1}{2}\myintom{\chi_h\xi(\epsilon^h |G^h|)|G^h|^2}\leq C\xi(\sqrt{\epsilon_h})$$ and therefore it is converging to zero as $h\ten 0$. Collecting the previous remarks, it follows that $$\liminf_{h \ten 0}\frac{\cal{J}^h(Y^h)}{\epsilon_h^2}\geq \frac{1}{2}\myintom{Q_3(G)}.$$
  
  We can decompose $G$ as 
  $$G=\Big(G-\myintr{G}\Big)+\myintr{G},$$
  where by the characterizations \eqref{tan} and \eqref{g11u}
  $$\Big(G-\myintr{G}\Big)_{tan}=-t\Big(\begin{array}{cc}0&w'\\w'&b\end{array}\Big)\quad\text{ and }\quad\myintr{G_{11}}=g.$$
  Therefore, by developing the quadratic form and using \eqref{defE} and \eqref{conrm}, we obtain
   \begin{eqnarray}
 \nonumber \myintom{Q_3(G)}&=&\myintom{Q_3\Big(G-\myintr{G}\Big)}+\myintss{Q_3\Big(\myintr{G}\Big)}\\
\nonumber&\geq& \frac{1}{12}\myintss{Q_2(s,w',b)}+\myintss{\mathbb{E}g^2}.
\end{eqnarray}
This concludes the proof.
\end{proof}
\section{Construction of the recovery sequence}
\label{construction}
 \noindent In this section  we show that the lower bound obtained in Theorem \ref{liminft} is optimal by exhibiting a recovery sequence.  The structure of such an optimal sequence varies according to the values of $\lambda$ and $\mu$. 
 \begin{teo}
 \label{limsupt}
 Assume \eqref{l} and \eqref{ls}. Let $\cal{A}_{\lambda,\mu}$ be the classes defined in \eqref{ainin}--\eqref{a0m}. Then, if $\mu>0$, for any $(w,g,b)\in \cal{A}_{\lambda,\mu}$ there exists a sequence of deformations $(y^h)\subset W^{1,2}(\Omega;\mathbb{R}^3)$ such that, defining $g^h$ and $w^h$ as in \eqref{defu} and \eqref{defw}, we have
\begin{eqnarray}
\label{cn0}&& y^h\ten x_1e_1\text{ strongly in }W^{1,2}(\Omega;\mathbb{R}^3),\\
 \label{cn1}&&g^h\ten g \text{ strongly in }L^2(\st),\\
  \label{cn3}&&w^h\ten w \text{ in }L^2(\st),\\
   \label{cn5}&&\frac{\partial_sw^h}{h}\ten b \text{ strongly in }L^2(\omega).
\end{eqnarray}
Moreover,
\begin{equation}
\label{cen}
\limsup_{h\ten 0}\frac{1}{\epsilon_h^2}\cal{J}^h(y^h)\leq \cal{J}_{\lambda,\mu}(w,u,b),
\end{equation}
where $\cal{J}_{\lambda,\mu}$ is the functional defined in \eqref{jlm}.

The same conclusion holds if $\mu=0$, assuming in addition the hypotheses of Lemma \ref{aplem2}.
\end{teo}
\begin{proof}
For the sake of simplicity, we divide the proof into five steps. In the first step we consider the case where $\lambda=+\infty$. Then we show how the recovery sequence must be modified for different values of $\lambda$ and $\mu$. \\
{\bf Step 1: $\lambda=\mu=+\infty$.}\\
Let $(w,g,b)\in \cal{A}_{\infty,\infty}$. We can assume that $w\in C^{\infty}([0,L])$, $b\in C^{\infty}(\lin{\omega})$, and there exist $\alpha_i\in C^{\infty}([0,L])$, $i=2,3,4$, such that
$$g=\alpha_2''\gamma_2+\alpha_3''\gamma_3+\alpha_4''.$$ 
The general case follows from approximation and standard arguments in $\Gamma$-convergence.

Let $\sigma_i\in C^{5}(\lin{\omega})$, $i=1,2,3$, be such that
\begin{equation}
\label{minpt}
Q_2(s,w,b)=Q_3\Bigg(R_0\Bigg(\begin{array}{ccc}0&w'&\sigma_1\\w'&b&\sigma_2\\\sigma_1&\sigma_2&\sigma_3\end{array}\Bigg)R_0^T\Bigg)
\end{equation}
for every $(x_1,s)\in\lin{\omega}$, and let $H\in C^{5}(\lin{\omega};\mthree_{\sym})$, $H=(h_{ij})$, be defined as 
\begin{equation}
\nonumber
H:=R_0\Bigg(\begin{array}{ccc}0&0&\sigma_1\\0&0&\sigma_2\\\sigma_1&\sigma_2&\sigma_3\end{array}\Bigg)R_0^T.
\end{equation}
For every $h>0$ we introduce the functions $\sigma^h\in C^5(\lin{\Omega};\mathbb{R}^3)$ defined as 
\begin{equation}
\nonumber
\sigma^h:=\epsilon_h\delta_h \Big(\frac{t^2}{2}-\frac{1}{24}\Big)\Bigg(\begin{array}{c}2\sigma_1\\2\sigma_2\tau_2-\sigma_3\tau_3\\2\sigma_2\tau_3+\sigma_3\tau_2\end{array}\Bigg).
\end{equation}
 It is easy to see that
\begin{equation}
\label{plh}
\sym(\nh\sigma^hR_0^T)=\epsilon_htH+o(\epsilon_h).
\end{equation}

Let also $F\in\mthree$ be the matrix defined by 
\begin{equation}
\label{plhbis}
\mathbb{E}=Q_3(e_1\otimes e_1+F),
\end{equation}
where $\mathbb{E}$ is the quantity introduced in \eqref{defE}. 

Finally, let $v\in C^6(\lin{\st};\mathbb{R}^2)$, $v=(v_2,v_3)$ be a solution of\begin{numcases}{}
 \partial_s v\cdot \lin{\tau}=0\text{ in }\omega,& \label{eku1}\\
 \partial_s\Big(\partial_sv\cdot \lin{n}\Big)=b\text{ in }\omega &\label{eku2}
\end{numcases}
and let $\lin{\psi}{}^{\frac{\delta_h}{h}}$ be the map introduced in \eqref{psie}, with $\epsilon=\frac{\delta_h}{h}$.

We consider the sequence
\begin{eqnarray}
\nonumber \capp{y}^h&=&\psi^h+\epsilon_h\Big(\begin{array}{c}\alpha_2'\\\alpha_3'\end{array}\Big)\cdot \lin{\psi}{}^{\frac{\delta_h}{h}}e_1+\epsilon_h\alpha_4'e_1-\frac{\epsilon_h}{h}\Bigg(\begin{array}{c}0\\\alpha_2\\\alpha_3\end{array}\Bigg)\\
\nonumber &+&\epsilon_h F\Big(h\Big(\alpha_4''\gamma+\sum_{i=2,3}\alpha_i''\int_0^s{\gamma_i(\xi)\tau(\xi) d\xi}\Big)+\delta_h t \Big(\alpha_4''+\sum_{i=2,3}\alpha_i''\gamma_i\Big)n\Big)\\
\nonumber &+&\frac{\epsilon_h}{\delta_h}w\Bigg(h\Bigg(\begin{array}{c}0\\-\gamma_3\\\gamma_2\end{array}\Bigg)-\delta_h t\tau\Bigg)-\frac{h\epsilon_h}{\delta_h}w' \Big(\delta_htT-h\int_0^s{N(\xi)d\xi}\Big)e_1\\ 
\nonumber &-&th\epsilon_h \Big(\partial_sv\cdot \lin{n}\Big) \tau+\frac{h^2\epsilon_h}{\delta_h}\Big(\begin{array}{c}0\\v\end{array}\Big)\\
\nonumber &-&\sigma^h-\frac{\epsilon_h^2}{2\delta_h^2}w^2(h\gamma+\delta_h t n).
\end{eqnarray}
We briefly comment on the structure of $\capp{y}^h$: the terms in the first line are related to conditions \eqref{cn0} and \eqref{cn1}, the second line is a corrective term to obtain the optimal constant $\mathbb{E}$, the terms in the third and the fourth line are introduced to satisfy respectively conditions \eqref{cn3} and \eqref{cn5}, and the last line contains a further corrective term.

We first prove that $\capp{y}^h$ satisfies \eqref{cn0}--\eqref{cn5}. By \eqref{l} we have
$$\|\capp{y}^h-x_1e_1\|_{W^{1,2}(\Omega;\mathbb{R}^3)}\leq Ch,$$
from which \eqref{cn0} follows. Condition \eqref{cn1} holds since
\begin{equation}
\label{cu0}
\partial_1(\capp{y}^h_1-x_1)=\epsilon_h g+\frac{h^2\epsilon_h}{\delta_h}w''\int_0^s{N(\xi)d\xi}+o(\epsilon_h)
\end{equation}
and $\lambda=+\infty$.
By the equality
\begin{equation}
\nonumber \frac{\delta_h}{h\epsilon_h} \myintr{\partial_s (\capp{y}^h-\psi^h)\cdot n}=w+h\partial_sv\cdot \n+o(h),
\end{equation}
and by \eqref{eku2}, we deduce \eqref{cn3} and \eqref{cn5}.

To prove convergence of the energies, we first compute the rescaled gradient of the deformations. By \eqref{eku1} and \eqref{eku2}, we obtain
\begin{eqnarray}
\nonumber \nh \capp{y}^h&=&R_0+\epsilon_h ge_1\otimes e_1
+\epsilon_h gF\big(0\big|\tau\big|n\big)\\
\nonumber &+&\frac{\epsilon_h}{h}\Bigg(\begin{array}{ccc}0&\alpha_2'\tau_2+\alpha_3'\tau_3&\alpha_3'\tau_2-\alpha_2'\tau_3\\-\alpha_2'&0&0\\-\alpha_3'&0&0\end{array}\Bigg)\\
\nonumber &-&\epsilon_ht\big(w'\tau\big|w'e_1+b\tau\big|0\big)+\Big(\frac{\epsilon_h}{\delta_h}w+\frac{h\epsilon_h}{\delta_h}(\partial_sv\cdot \lin{n})n\Big)\big(0\big|n\big|-\tau\big)\\\nonumber&+&\frac{h\epsilon_h}{\delta_h}w'\Bigg(\begin{array}{ccc}0&N&-T\\-\gamma_3&0&0\\\gamma_2&0&0\end{array}\Bigg)-\nh \sigma^h-\frac{\epsilon_h^2}{2\delta_h^2}w^2\big(0\big|\tau\big|n\big)+o(\epsilon_h). 
\end{eqnarray}

We point out that the two terms 
$$\Big(\frac{h^2\epsilon_h}{\delta_h}w''\int_0^s{N(\xi)d\xi}\Big)e_1\otimes e_1\quad\text{ and }\quad\frac{h^2\epsilon_h}{\delta_h}\Bigg(\begin{array}{c}0\\\partial_1 v_2\\\partial_1 v_3\end{array}\Bigg)\otimes e_1$$
are infinitesimal of order larger than $\epsilon_h$ since we are assuming $\lambda=+\infty$. Therefore they can be included in the error term $o(\epsilon_h)$.

 The previous equality in turn gives:
\begin{eqnarray}
\nonumber \nh \capp{y}^hR_0^T&=&Id+\epsilon_h g(e_1\otimes e_1+F)+\frac{\epsilon_h}{h}\Bigg(\begin{array}{ccc}0&\alpha_2&\alpha_3\\-\alpha_2&0&0\\-\alpha_3&0&0\end{array}\Bigg)\\
\nonumber &-&\epsilon_ht\big(w'\tau\big|w'e_1+b\tau\big|0\big)R_0^T+\Big(\frac{\epsilon_h}{\delta_h}w+\frac{h\epsilon_h}{\delta_h}(\partial_sv\cdot\lin{n})n\Big)\Bigg(\begin{array}{ccc}0&0&0\\0&0&-1\\0&1&0\end{array}\Bigg)\\\nonumber&+&\frac{h\epsilon_h}{\delta_h}w'\Bigg(\begin{array}{ccc}0&\gamma_3&-\gamma_2\\-\gamma_3&0&0\\\gamma_2&0&0\end{array}\Bigg)-\nh \sigma^hR_0^T-\frac{\epsilon_h^2}{2\delta_h^2}w^2\Bigg(\begin{array}{ccc}0&0&0\\0&1&0\\0&0&1\end{array}\Bigg)\\
\nonumber &+&o(\epsilon_h). 
\end{eqnarray}
The identity $(Id+F)^T(Id+F)=Id+2\sym F+F^TF$ yields
\begin{eqnarray}
\nonumber (\nh \capp{y}^hR_0^T)^T(\nh \capp{y}^hR_0^T)=Id+2\epsilon_h M+o(\epsilon_h),
\end{eqnarray}
where $M$ is given by
\begin{equation}
\nonumber
M:=g(e_1\otimes e_1+\sym F)-t\Bigg(R_0\Bigg(\begin{array}{ccc}0&w'&0\\w'&b&0\\0&0&0\end{array}\Bigg)R_0^T+H\Bigg),
\end{equation}
owing to \eqref{plh}. Hence, by frame-indifference,
$$W(\nh \capp{y}^hR_0^T)=W\Big(\sqrt{(\nh \capp{y}^hR_0^T)^T(\nh \capp{y}^hR_0^T)}\Big)=W(Id+\epsilon_h M+o(\epsilon_h)).$$
Since $M$ is bounded in $L^{\infty}$, there exists $\lin{h}$ such that if $h<\lin{h}$, $Id+\epsilon_h M+o(\epsilon_h)$ belongs to the neighbourhood of SO(3) where $W$ is $C^2$, therefore a Taylor expansion around the identity gives:
$$\frac{1}{\epsilon_h^2}W(\nh \capp{y}^hR_0^T)\ten \frac{1}{2} Q_3(M)\text{ pointwise },$$
and
$$W(\nh \capp{y}^hR_0^T)\leq C(|M|^2+1),$$
for some constant $C$.
By dominated convergence theorem and by \eqref{minpt} and \eqref{plhbis} we deduce
\begin{eqnarray}
\nonumber \lim_{h\ten 0}\frac{\cal{J}^h(\capp{y}^h)}{\epsilon_h^2}&=&\frac{1}{2}\myintom{Q_3(M)}\\
\nonumber &=&\frac{1}{24}\myintss{Q_2(s,w',b)}+\frac{1}{2}\myintss{\mathbb{E}g^2},
\end{eqnarray}
 which concludes the proof of \eqref{cen} in the case where $\lambda=+\infty$. \\
{\bf Step 2: $0<\lambda<+\infty$ and $\mu=+\infty$.}\\
Let $(w,g,b)\in\cal{A}_{\lambda,\infty}$. We can assume that $w\in C^{\infty}[0,L]$, $b\in C^{\infty}(\lin{\omega})$, and there exist $\alpha_i\in C^{\infty}(0,L)$, $i=2,3,4$, such that
 $$g=\frac{1}{\lambda}w''\int_0^s{N(\xi)d\xi}+\alpha_2''\tau_2+\alpha_3''\tau_3+\alpha_4''.$$ 
 Let $v$ be defined as in \eqref{eku1}--\eqref{eku2} and let $u\in C^6(\lin{\omega})$ be such that \mbox{$\partial_s u+\partial_1v\cdot \lin{\tau}=0$} in $\omega$. 
 
 We consider the sequence
\begin{eqnarray}
\nonumber y^h&=&\capp{y}^h+\frac{h^2\epsilon_h}{\delta_h}F\Big(hw''\int_0^s\Big({\int_0^{\xi}{N(\eta)d\eta}\Big)\tau(\xi) d\xi}+\delta_h t w''\Big(\int_0^s{N(\xi) d\xi}\Big) n\Big)\\
\nonumber &+&\frac{h^3\epsilon_h}{\delta_h}\Big(u-\frac{\delta_h}{h}t\partial_1 v\cdot \n\Big)e_1,
\end{eqnarray}
which is obtained adding to the sequence $(\capp{y}^h)$ introduced in Step 1 two corrective terms. The first corrective term is due to the different structure of $g$, while the second one is needed to cancel the contribution to the energy of the quantity $$\frac{h^2\epsilon_h}{\delta_h}\Bigg(\begin{array}{c}0\\\partial_1 v_2\\\partial_1 v_3\end{array}\Bigg)\otimes e_1,$$
which is now of order $\epsilon_h$. We observe that the term $\big(\frac{h^2\epsilon_h}{\delta_h}w''\int_{0}^s{N(\xi)d\xi}\big)e_1\otimes e_1$ is now included in the expression of $g$.

The proof of \eqref{cn0}--\eqref{cn5} is analogous to the one in Step 1. To prove convergence of the energies, we argue as in Step 1 and we deduce 
$$\lim_{h\ten 0}{\frac{\cal{J}^h(y^h)}{\epsilon_h^2}}=\frac{1}{24}\myintss{Q_2(s,w',b)}+\frac{1}{2}\myintss{\mathbb{E}g^2}.$$
 A standard approximation argument leads then to the conclusion.\\
 \begin{comment}% Convolvo le w con mollificatori prima variabile, convolvo le f idem, definisco le u come d1fk e convoluzione altri pezzi.
 \end{comment}
{\bf Step 3: $\lambda=0$ and $\mu=+\infty$}.\\
Let $(w,g,b)\in\cal{A}_{0,\infty}$. Then $w$ is affine. Moreover, we can assume that $b\in C^{\infty}(\lin{\st})$, and there exist $\alpha_i\in C^{\infty}[0,L]$, $i=1,\cdots,4$, such that
$$g=\alpha_1''\int_0^s{Nd\xi}+\alpha_2''\gamma_2+\alpha_3''\gamma_3+\alpha_4''.$$
 
 Let $v$ and $u$ be defined as in the previous step. We consider the sequence: 
 \begin{eqnarray}
\nonumber {y}^h&=&\capp{y}^h+\epsilon_h \alpha_1' \int_0^s{N(\xi) d\xi}e_1-\frac{\epsilon_h\delta_h t}{h}\alpha_1'Te_1+\frac{\epsilon_h}{h}\alpha_1\Bigg(\begin{array}{c}0\\-\gamma_3\\\gamma_2\end{array}\Bigg)-\frac{\epsilon_h\delta_h t}{h^2}\alpha_1 \tau\\
\nonumber &+&\epsilon_h F\Big(h\Big(\alpha_1''\int_0^s{\Big(\int_0^{\xi}{N(\eta)d\eta}\Big)\tau(\xi) d\xi}\Big)+\delta_h t \alpha_1''\Big(\int_0^s{N(\xi)d\xi}\Big)n\Big)\\
\nonumber &+&\frac{h^3\epsilon_h}{\delta_h}\Big(u-\frac{\delta_h}{h} t\partial_1v\cdot \lin{n}\Big)e_1,
\end{eqnarray}
where $(\capp{y}^h)$ is the sequence introduced in Step 1.

We observe that the previous sequence is obtained by a slight modification of the recovery sequence introduced in Step 2, due to the fact that, since $\lambda=0$, the contribution of $w''$ to the energy is zero and the role of $w''$ in the structure of $g$ is now played by $\alpha_1''$.

Arguing as in Step 1, it is straightforward to prove \eqref{cn0}--\eqref{cn5}. The same computations of Step 1 yield also convergence of the energies and the conclusion follows by approximation.\\
{\bf Step 4: $\lambda=0$ and $0<\mu<+\infty$.}\\
Let $(w,g,b)\in\cal{A}_{0,\mu}$. Then $w$ is affine. Moreover, by Lemma \ref{aplem} we can reduce to the case where $g\in C^4(\lin{\st})$, \mbox{$b\in C^{3}(\lin{\st})$}, and there exists $\phi\in C^5(\lin{\st};\mathbb{R}^3)$ such that
\begin{equation}
\nonumber
\partial_1 \phi_1=\mu g, \quad \partial_s \phi \cdot \tau=0, \quad \partial_s \phi_1+\partial_1 \phi\cdot \tau=0, \quad\text{ and }\quad
\partial_s(\partial_s \phi \cdot n)=b . 
\end{equation} 

We define
\begin{eqnarray}
\nonumber y^h&:=&\psi^h+\frac{h^3\epsilon_h}{\delta_h} \phi_1 e_1+
\epsilon_hF\Big(h\int_0^s{g\tau d\xi}+\delta_h t g n\Big)\\
\nonumber &+&\epsilon_h\Bigg(-tw\tau+\frac{h}{\delta_h}w\Bigg(\begin{array}{c}0\\-\gamma_3\\\gamma_2\end{array}\Bigg)\Bigg)-\epsilon_h\Big(thw'T-\frac{h^2}{\delta_h}w'\int_0^s{N d\xi}\Big)e_1\\
\nonumber &-& th\epsilon_h(\partial_s \phi\cdot n)\tau+\frac{h^2\epsilon_h}{\delta_h}\Bigg(\begin{array}{c}0\\\phi_2\\\phi_3\end{array}\Bigg)-h^2\epsilon_h t\partial_1 \phi\cdot n e_1\\
\nonumber &-&\sigma^h-\frac{\epsilon_h^2}{2\delta_h^2}w^2(h\gamma+\delta_h t n),
\end{eqnarray}
where the terms in the first line are related to conditions \eqref{cn0} and \eqref{cn1} and to the optimal constant $\mathbb{E}$, whereas the second and the third lines are related to conditions \eqref{cn3} and \eqref{cn5} and to the quadratic form $Q_2$.

Arguing as in the previous steps it is straightforward to prove that conditions \eqref{cn0}--\eqref{cn5} are satisfied and that\begin{eqnarray}
\nonumber \lim_{h\ten 0}{\frac{\cal{J}^h(y^h)}{\epsilon_h^2}}=\frac{1}{24}\myintss{Q_2(s,w',b)}+\frac{1}{2}\myintss{\mathbb{E}g^2}.\end{eqnarray}
{\bf Step 5: $\lambda=\mu=0$.}\\
Assume that there exists a finite number of points $0=p_0<p_1<\cdots<p_m=1$ such that for every $i=0,\cdots,m-1$ we have that $k(s)>0$ for every $s\in (p_i,p_{i+1})$, or $k(s)<0$ for every $s\in (p_i,p_{i+1})$ or $k(s)=0$ for every $s\in (p_i,p_{i+1})$.

Let $(w,g,b)\in \cal{A}_{0,0}$. Then $w$ is affine. Moreover, by Remark \ref{recuse}, we can reduce to the case where $g\in C^4(\lin{\omega})$ and there exist two maps $u\in C^{6}(\lin{\omega})$ and $z\in C^5(\lin{\omega})$ such that $\partial_1^2 u=g$ and $\partial_s^2 u=kz$. By Lemma \ref{aplem2} we can also assume that $b\in C^3(\lin{\omega})$ and there exists $\phi\in C^5(\lin{\omega};\mathbb{R}^3)$ such that
\begin{equation}
\nonumber \partial_1\phi_1=0,\quad \partial_s \phi\cdot \tau=0,\quad \partial_s \phi_1+\partial_1\phi\cdot \tau=0\quad\text{ and }\quad \partial_s(\partial_s \phi\cdot n)=b.
\end{equation}

We define:
\begin{eqnarray}
\nonumber y^h&:=&\psi^h+\epsilon_h \Big(\partial_1u+\frac{\delta_h}{h}t\partial_1 z\Big)e_1-\frac{\epsilon_h}{h}(\partial_s u\tau+z n)+\frac{\epsilon_h\delta_h}{h^2}t(\partial_s u k+\partial_s z)\tau\\
\nonumber &+&
\epsilon_hF\Big(h\int_0^s{g\tau d\xi}+\delta_h t g n\Big)\\
\nonumber &+&\epsilon_h\Bigg(-tw\tau+\frac{h}{\delta_h}w\Bigg(\begin{array}{c}0\\-\gamma_3\\\gamma_2\end{array}\Bigg)\Bigg)-\epsilon_h\Big(thw'T-\frac{h^2}{\delta_h}w'\int_0^s{N d\xi}\Big)e_1\\
\nonumber &-& th\epsilon_h(\partial_s \phi\cdot n)\tau+\frac{h^2\epsilon_h}{\delta_h}\Bigg(\begin{array}{c}0\\\phi_2\\\phi_3\end{array}\Bigg)-h^2\epsilon_h t\partial_1 \phi\cdot n e_1+\frac{h^3\epsilon_h}{\delta_h}\phi_1 e_1\\
\nonumber &-&\sigma^h-\frac{\epsilon_h^2}{2\delta_h^2}w^2(h\gamma+\delta_h t n),
\end{eqnarray}
where the first line contains now some corrective terms to compensate the contribution given by $\partial_s u$, and the terms in the other lines play the same role as in the previous steps.

Arguing as in Step 1, it is immediate to prove \eqref{cn0}--\eqref{cn5}. The same computations of Step 1 yield also \eqref{cen}. Hence, the proof of the Theorem is completed.
\end{proof}
\bigskip
\bigskip

\noindent
\textbf{Acknowledgements.}
I warmly thank Maria Giovanna Mora for having proposed to me the study of this problem and for many helpful and interesting discussions and suggestions.\\
This work was partially supported by MIUR under PRIN 2008 and by INDAM under GNAMPA projects 2010 and 2011.

\end{document}